\documentclass[letterpaper, 10 pt, conference]{IEEEtran}  %

\IEEEoverridecommandlockouts                              %

\def\baselinestretch{0.999}

\newcommand{\vect}[1]{\mathbf{#1}}
\newcommand{\mat}[1]{\mathbf{#1}}

\newcommand{\diffs}[3]{\frac{\partial^2 #1}{
\ifx#2#3 
\partial #2^2
\else
\partial #2 \partial #3
\fi
}}

\makeatletter

\let\proof\@undefined
\let\endproof\@undefined
\makeatother

\usepackage{epsfig} %
\usepackage{amsmath} %
\usepackage{amssymb}  %
\usepackage{amsthm}
\usepackage{amsfonts}
\usepackage{mathtools}

\usepackage{paralist}

\newtheorem{prop}{Proposition}

\newtheorem{thm}{Theorem}
\newtheorem{remark}{Remark}
\newtheorem{lem}[thm]{Lemma}

\theoremstyle{definition}
\newtheorem{dfn}{Definition}
\theoremstyle{remark}

\usepackage{verbatim}

\usepackage{mathptmx}
\usepackage{times}
\usepackage{bm}

\usepackage{booktabs}
\usepackage{epstopdf}

\usepackage[norelsize,ruled,vlined,linesnumbered]{algorithm2e}
\SetInd{0.2em}{0.7em}
\SetCommentSty{textit}
\SetAlgoInsideSkip{smallskip}
\DontPrintSemicolon

\usepackage[ruled,vlined,linesnumbered]{algorithm2e}
\usepackage[noend]{algorithmic}

\usepackage[us]{datetime}

\usepackage[usenames,dvipsnames,table]{xcolor}

\usepackage{hyperref} %
\hypersetup{
    colorlinks,
    citecolor=black,
    filecolor=black,
    linkcolor=black,
    urlcolor=black,
    pdfauthor={},
    pdfsubject={},
    pdftitle={}
}

\usepackage{todonotes}

\usepackage{graphicx}

\usepackage{latexsym}
\usepackage{color}

\usepackage{cite}
\usepackage[per-mode=symbol,binary-units]{siunitx}

\def\figExpOnePosition{\includegraphics[width=1.0\columnwidth]{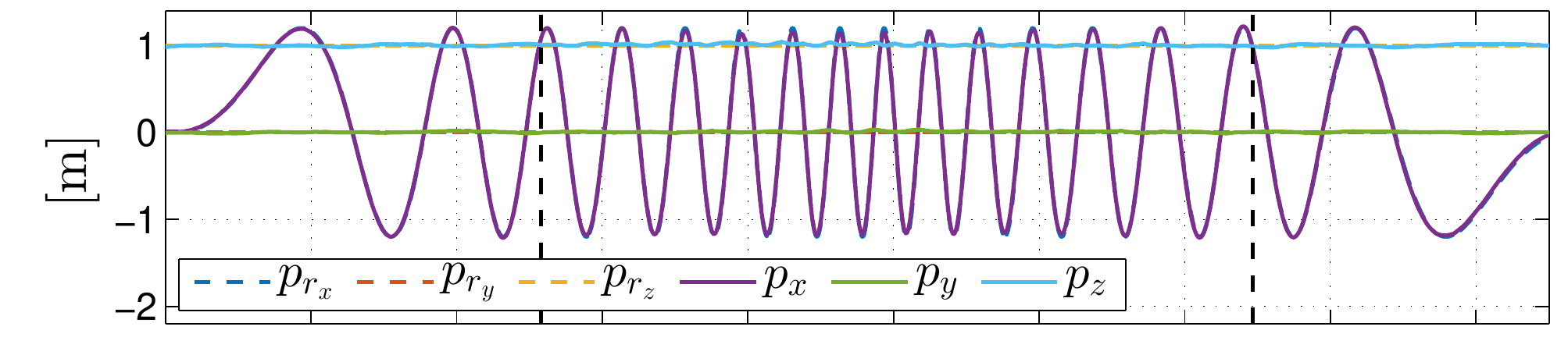}}
\def\figExpOneOrientation{\includegraphics[width=1.0\columnwidth]{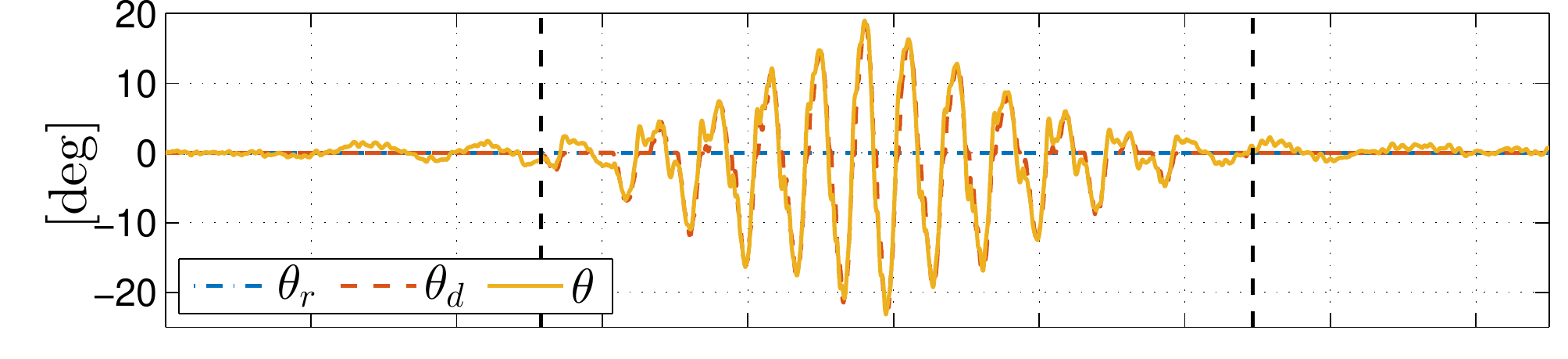}}
\def\figExpOnePosError{\includegraphics[width=1.0\columnwidth]{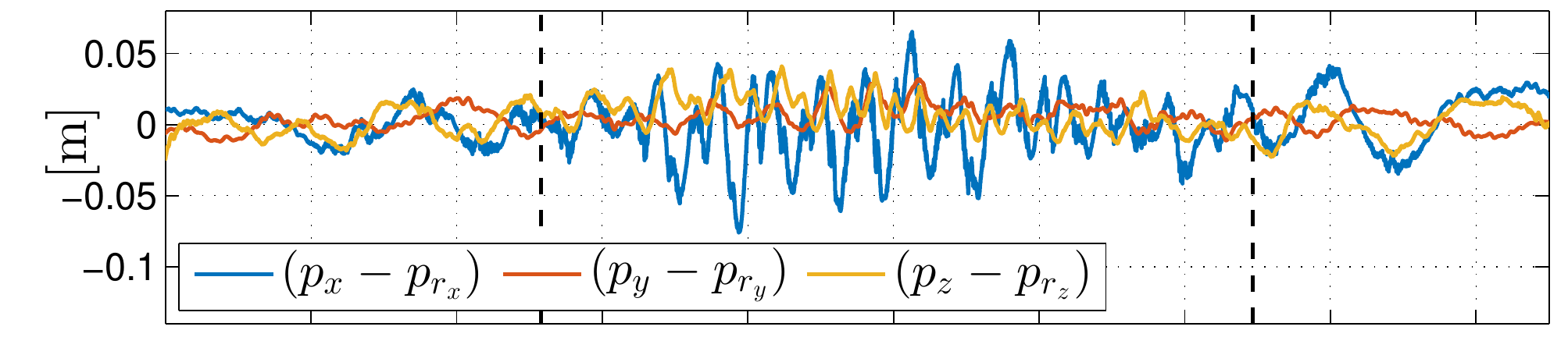}}
\def\figExpOneRotError{\includegraphics[width=1.0\columnwidth]{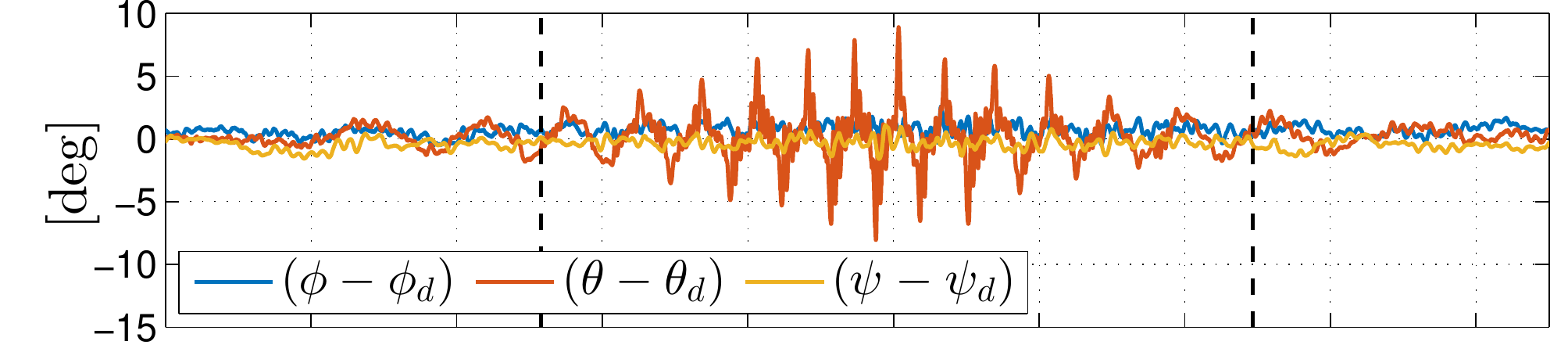}}
\def\figExpOneDesW{\includegraphics[width=1.0\columnwidth]{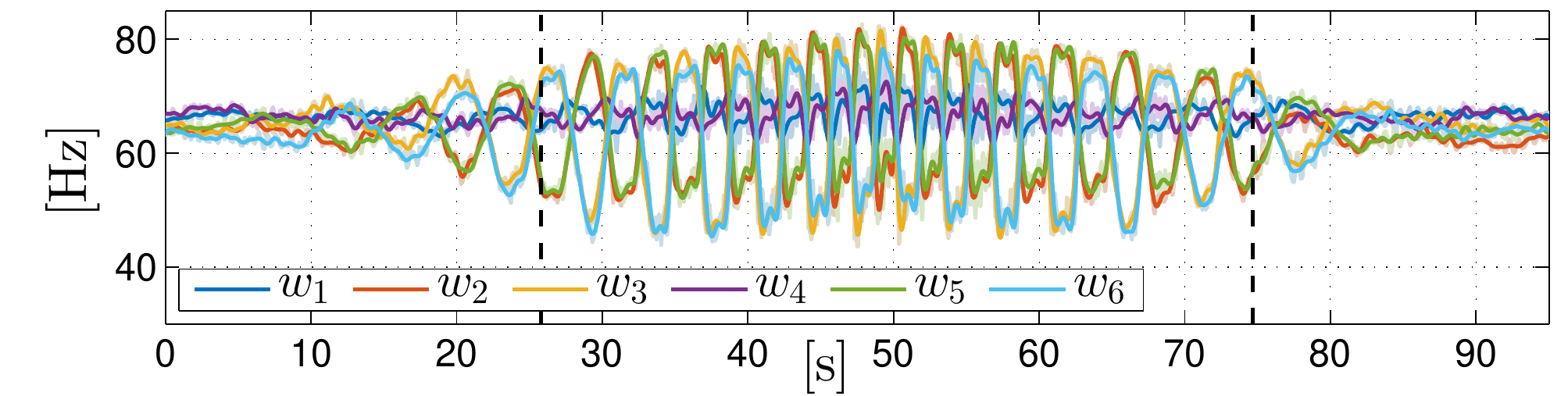}}
\def\figExpOneUOne{\includegraphics[width=1.0\columnwidth]{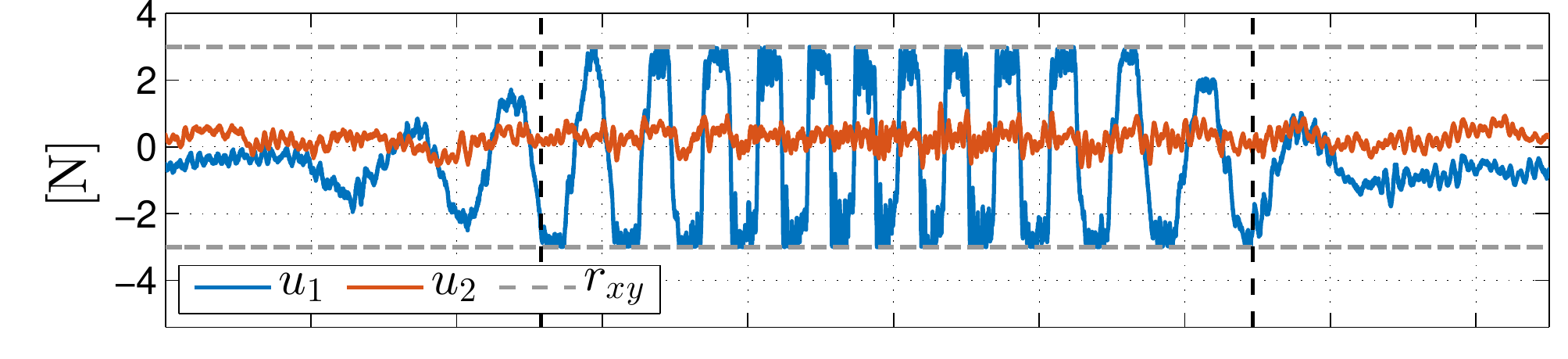}}
\def\figExpOneVelocity{\includegraphics[width=1.0\columnwidth]{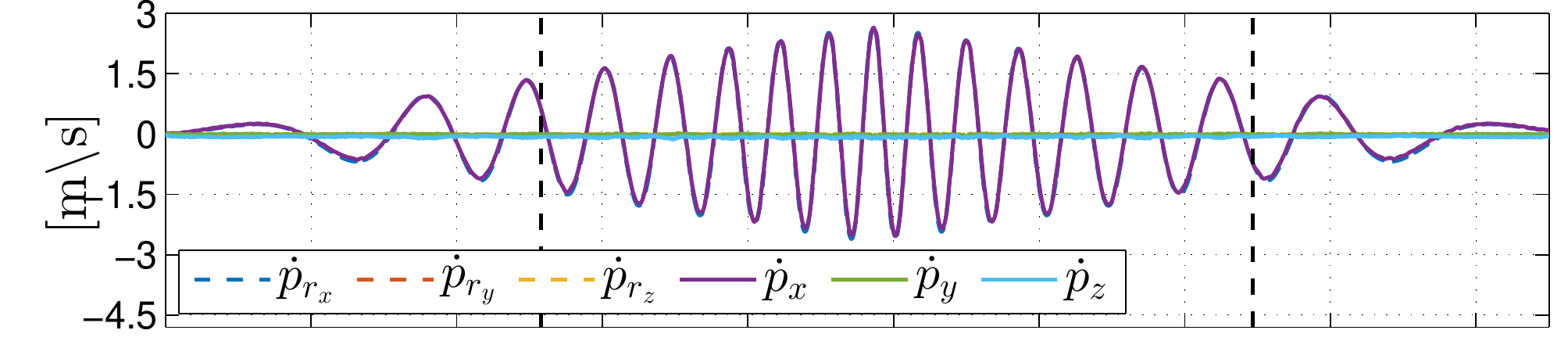}}
\def\figExpOneAcceleration{\includegraphics[width=1.0\columnwidth]{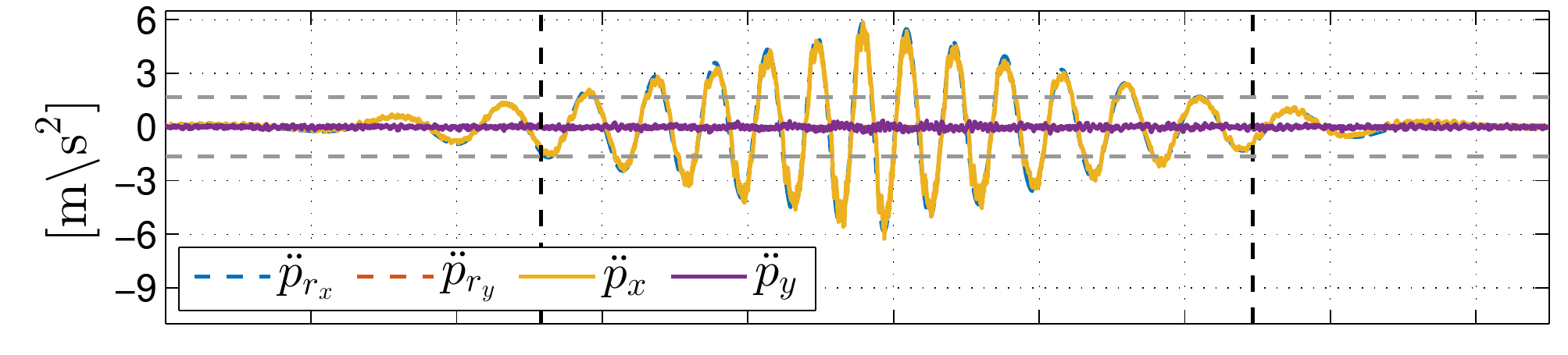}}

\def\figExpOneTwoPosition{\includegraphics[width=1.0\columnwidth]{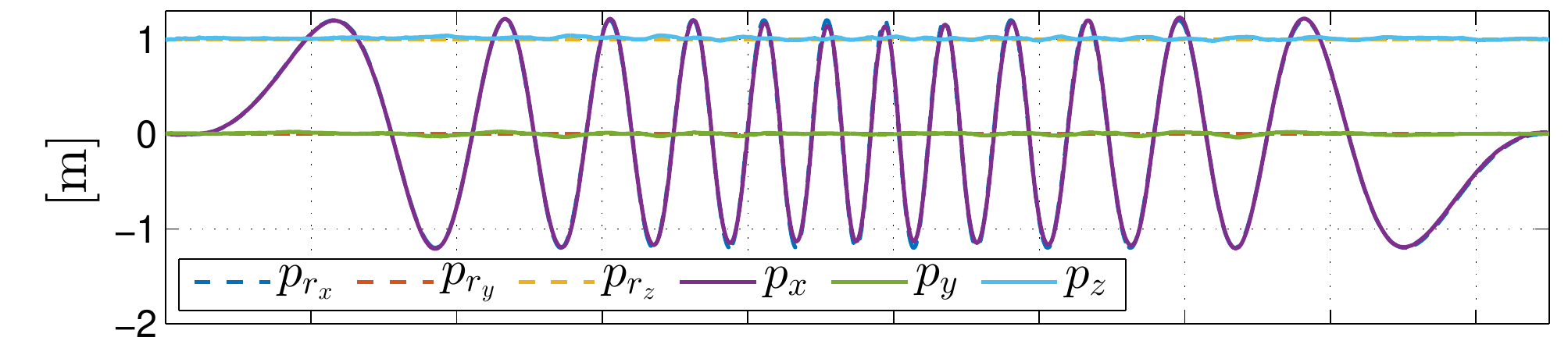}}
\def\figExpOneTwoOrientation{\includegraphics[width=1.0\columnwidth]{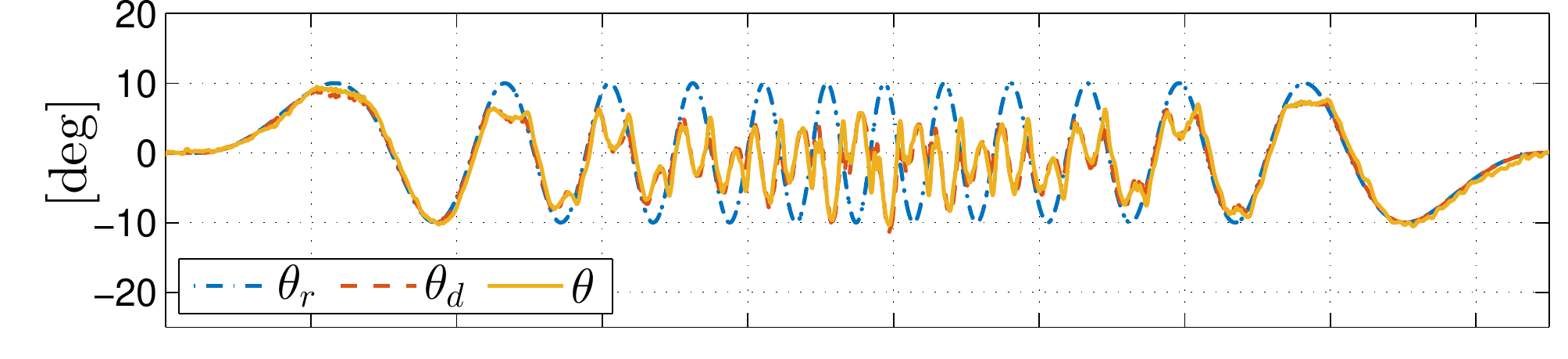}}
\def\figExpOneTwoPosError{\includegraphics[width=1.0\columnwidth]{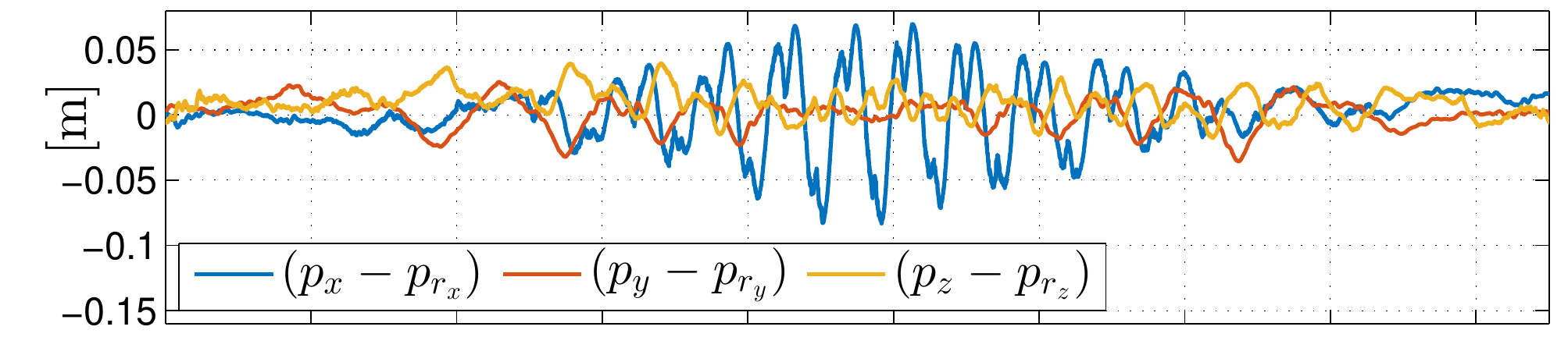}}
\def\figExpOneTwoRotError{\includegraphics[width=1.0\columnwidth]{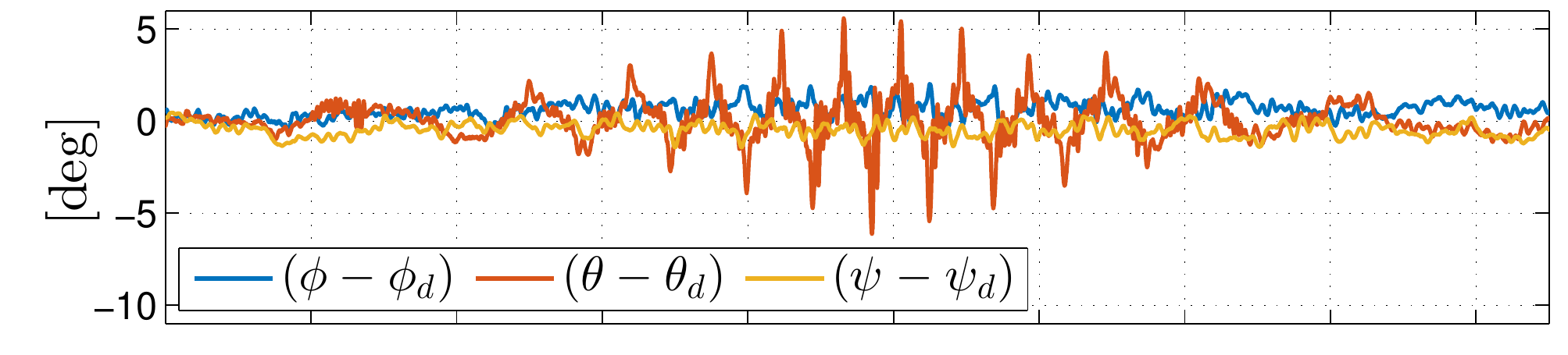}}
\def\figExpOneUOneTwo{\includegraphics[width=1.0\columnwidth]{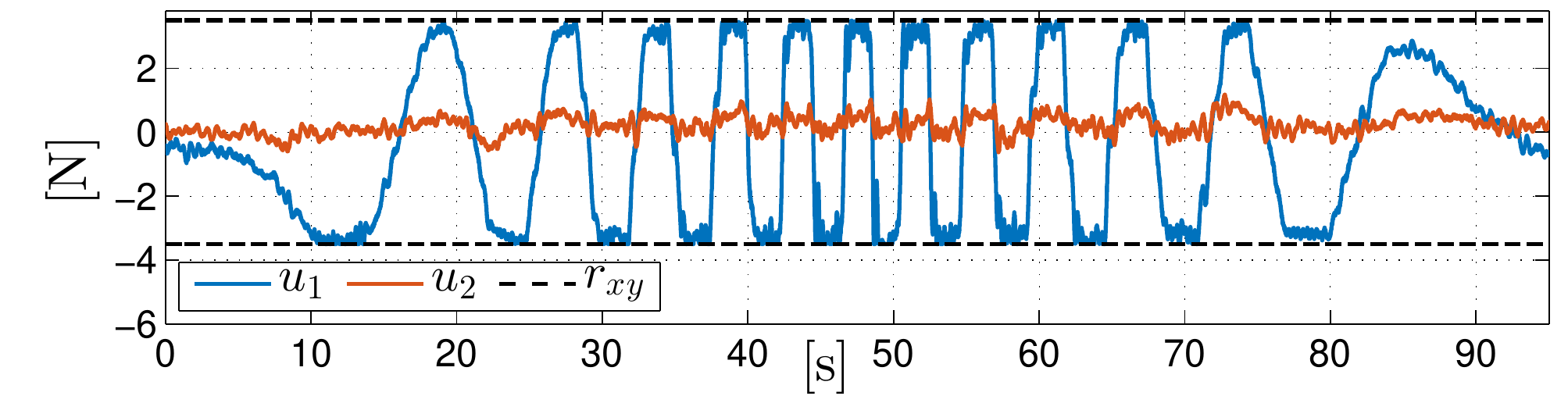}}
\def\figExpOneOrientationZero{\includegraphics[width=1.0\columnwidth]{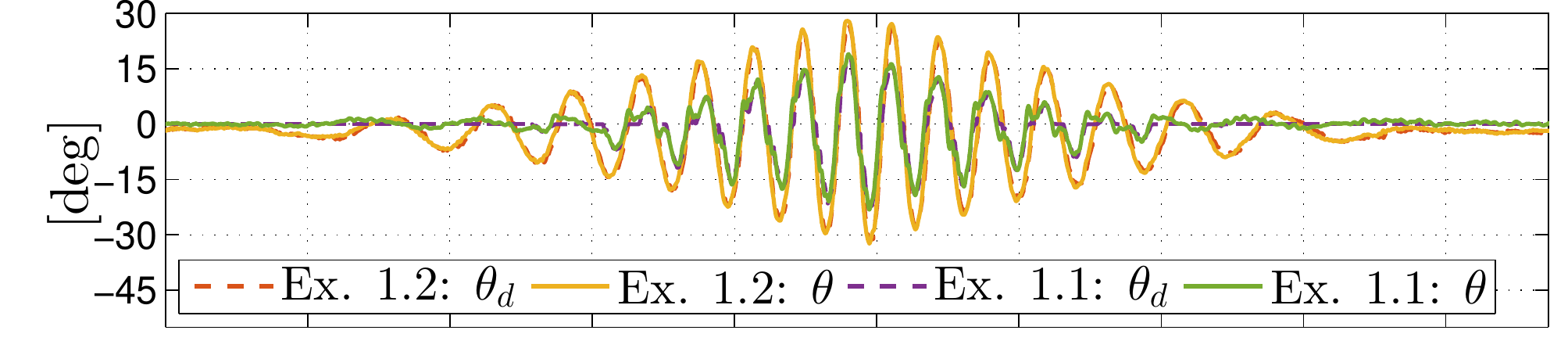}}
\def\figExpOnePosErrorZero{\includegraphics[width=1.0\columnwidth]{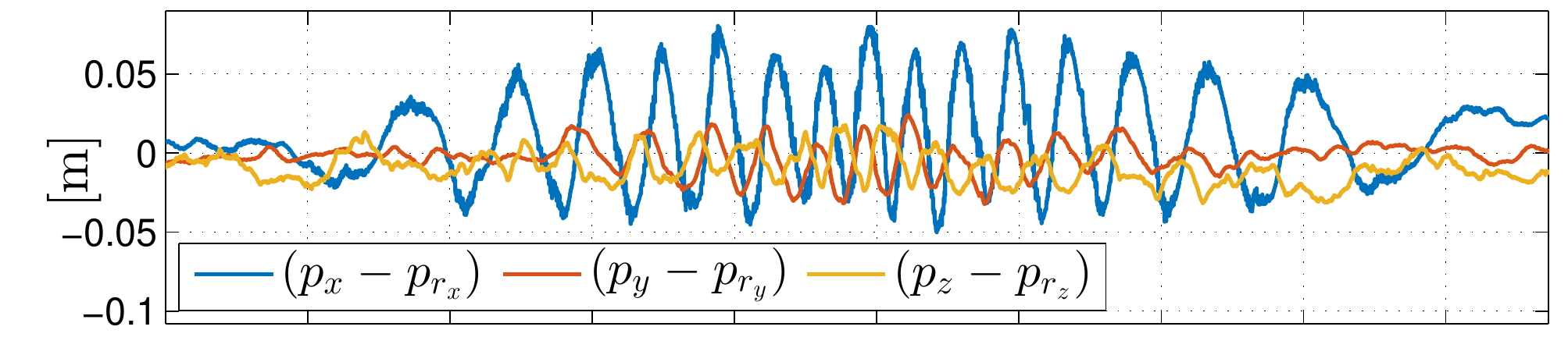}}
\def\figExpOneDesWZero{\includegraphics[width=1.0\columnwidth]{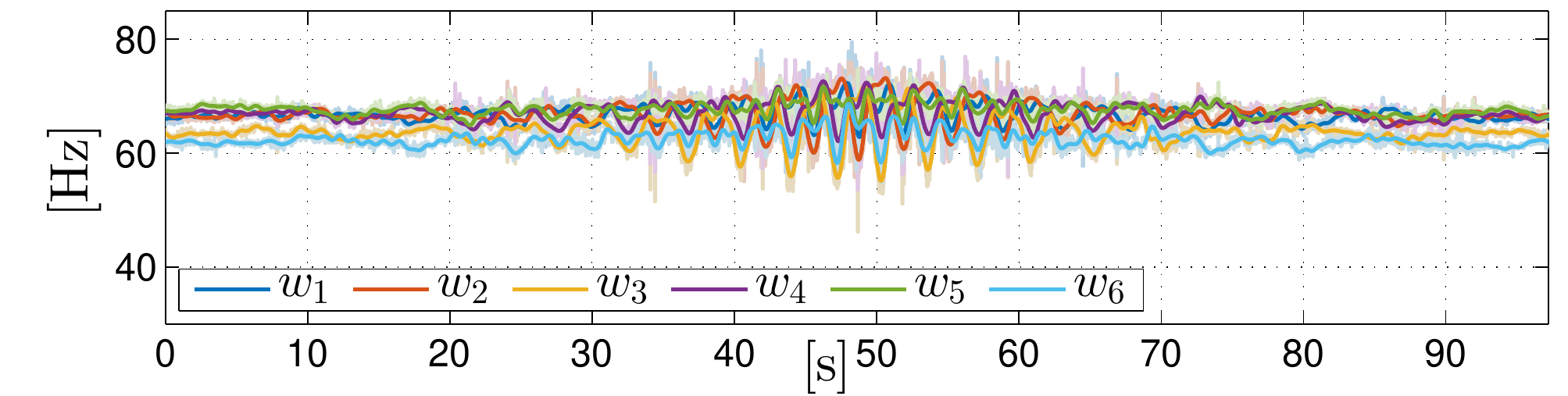}}
\def\figExpOneUOneZero{\includegraphics[width=1.0\columnwidth]{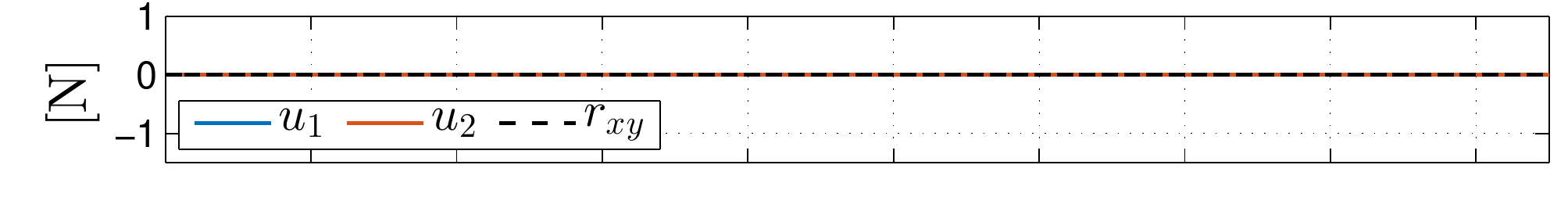}}

\def\figExpOnePositionSat{\includegraphics[width=1.0\columnwidth]{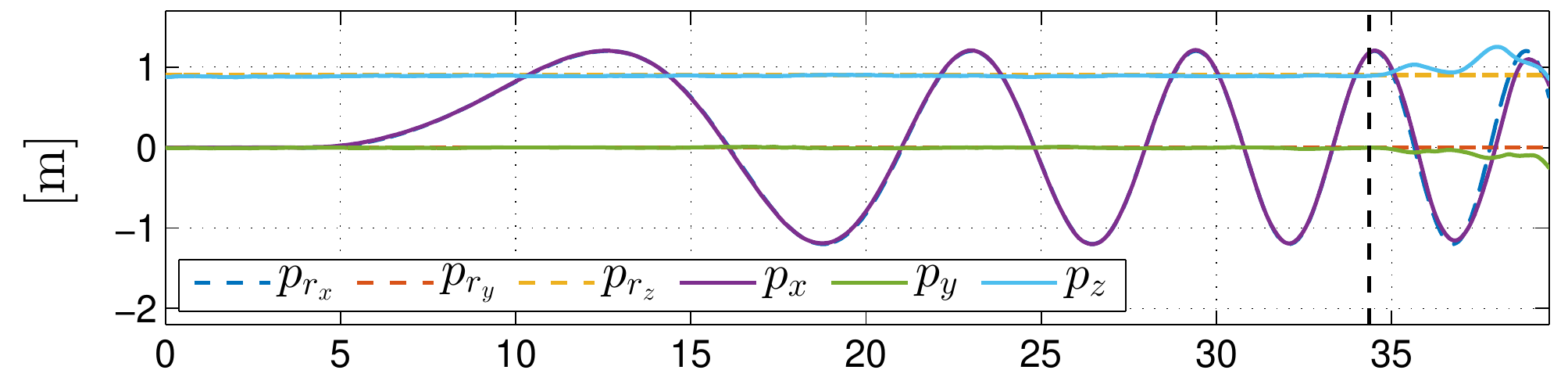}}
\def\figExpOneOrientationSat{\includegraphics[width=1.0\columnwidth]{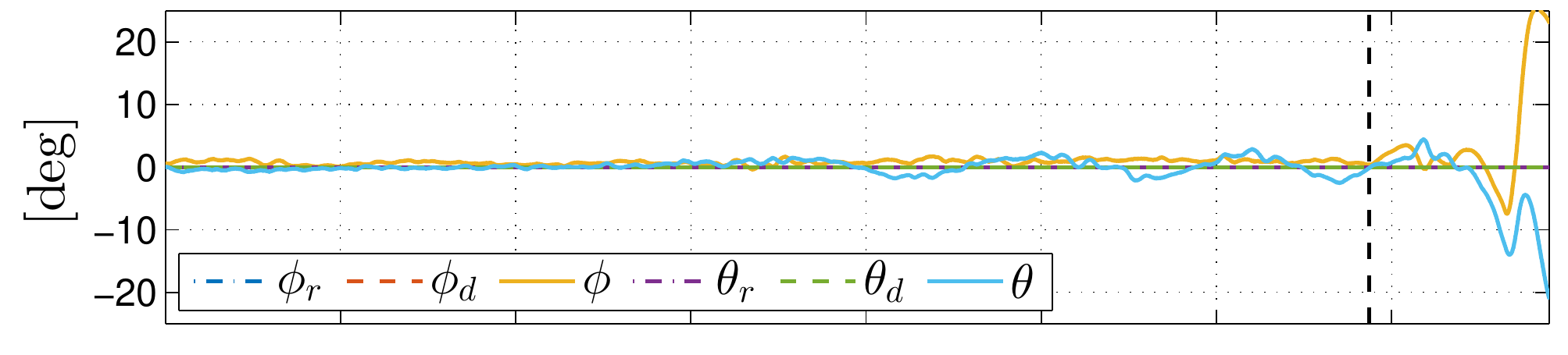}}
\def\figExpOnePosErrorSat{\includegraphics[width=1.0\columnwidth]{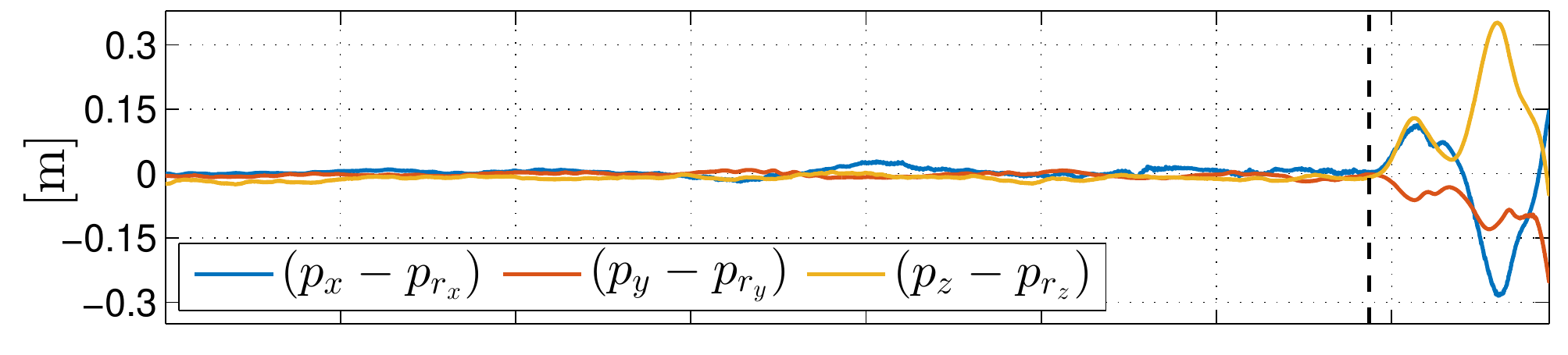}}
\def\figExpOneRotErrorSat{\includegraphics[width=1.0\columnwidth]{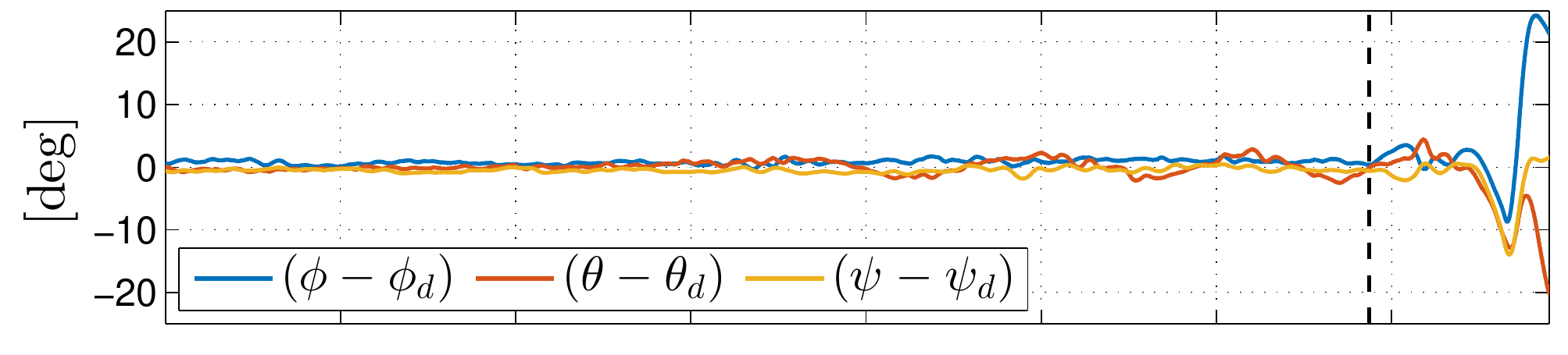}}
\def\figExpOneDesWSat{\includegraphics[width=1.0\columnwidth]{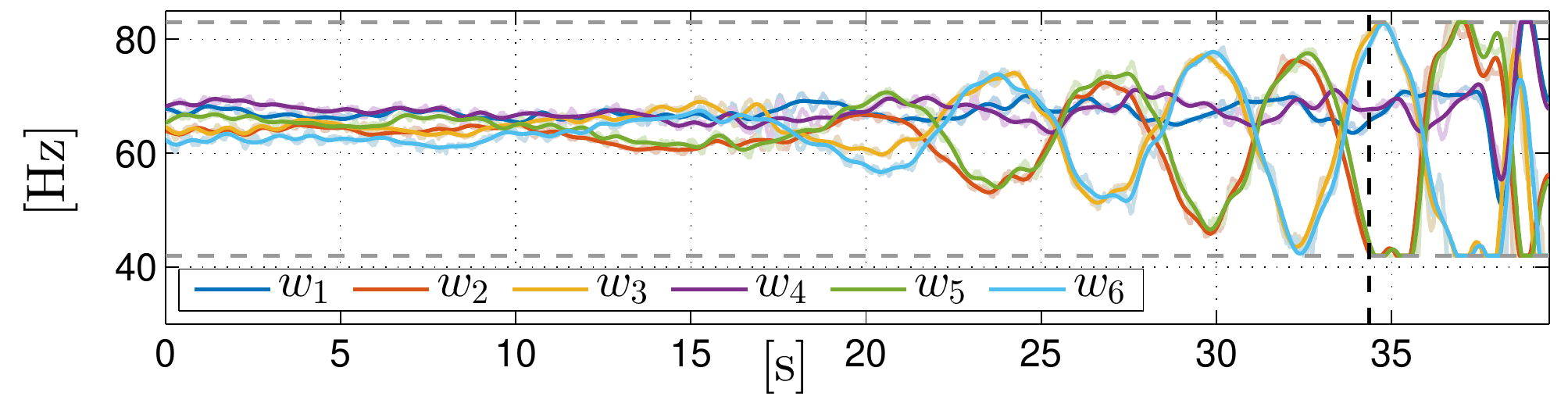}}

\def\figExpTwoPosition{\includegraphics[width=1.0\columnwidth]{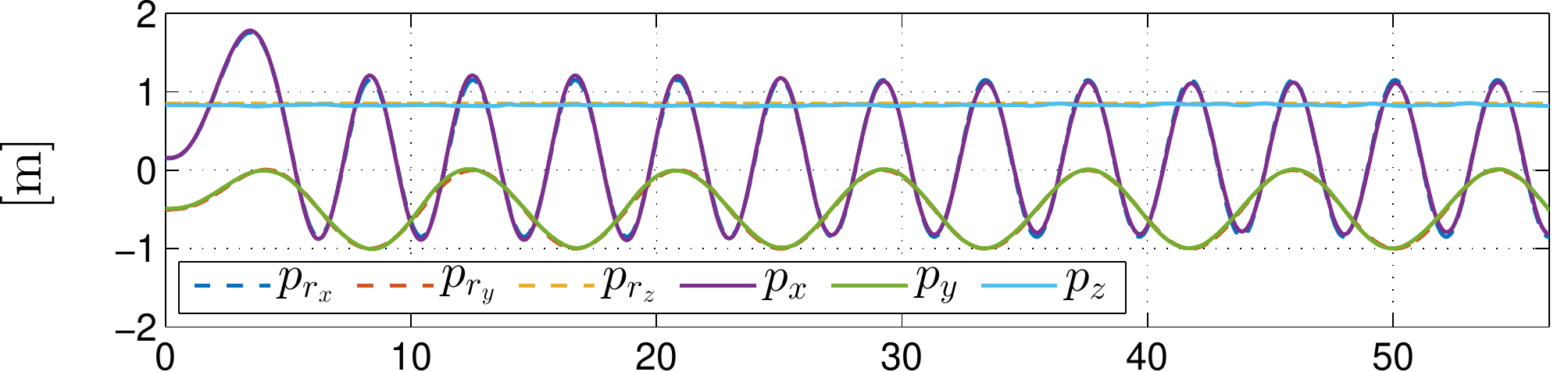}}
\def\figExpTwoOrientation{\includegraphics[width=1.0\columnwidth]{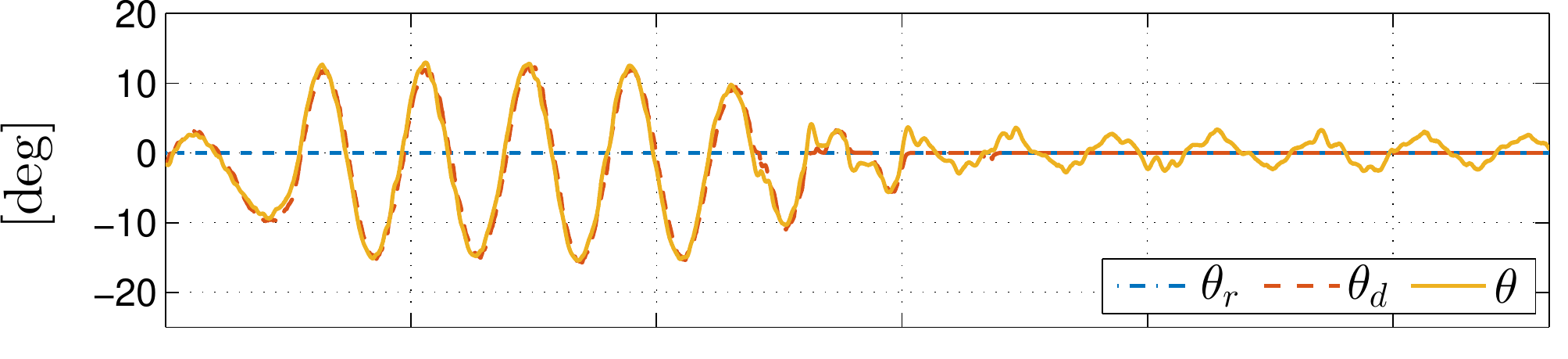}}
\def\figExpTwoPosError{\includegraphics[width=1.0\columnwidth]{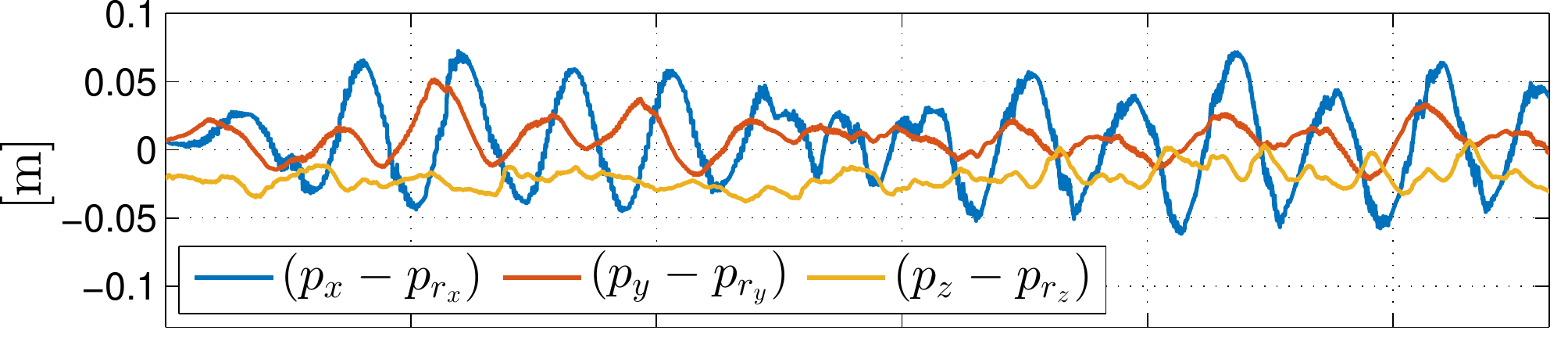}}
\def\figExpTwoRotError{\includegraphics[width=1.0\columnwidth]{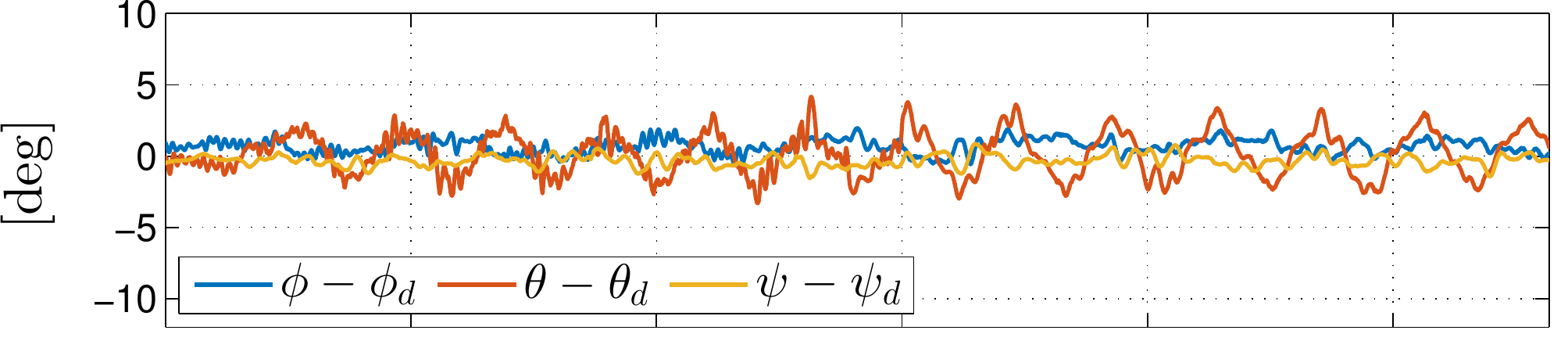}}
\def\figExpTwoDesW{\includegraphics[width=1.0\columnwidth]{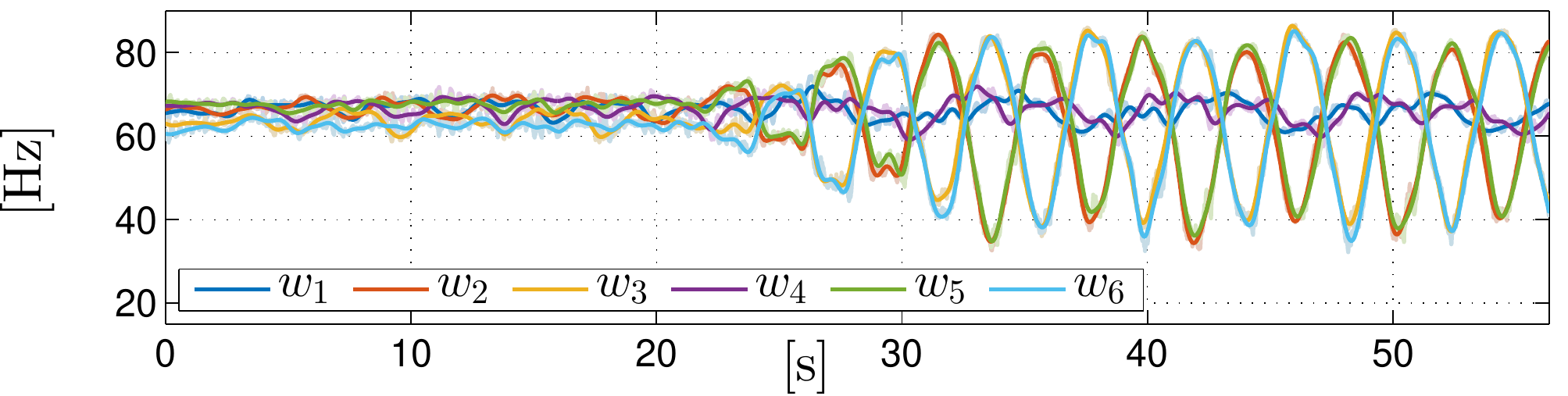}}
\def\figExpTwoUOne{\includegraphics[width=1.0\columnwidth]{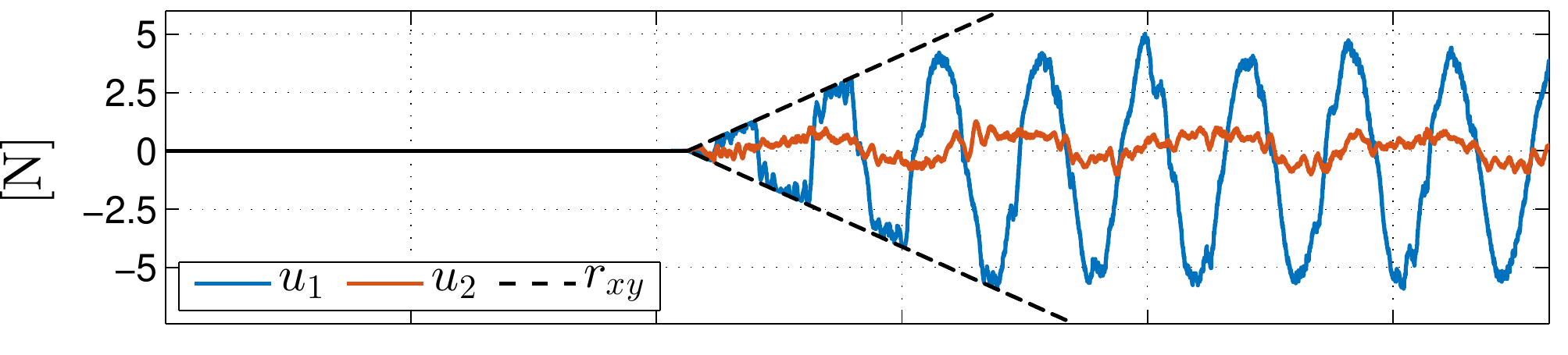}}

\def\figForceVolume{\includegraphics[width=.55\columnwidth]{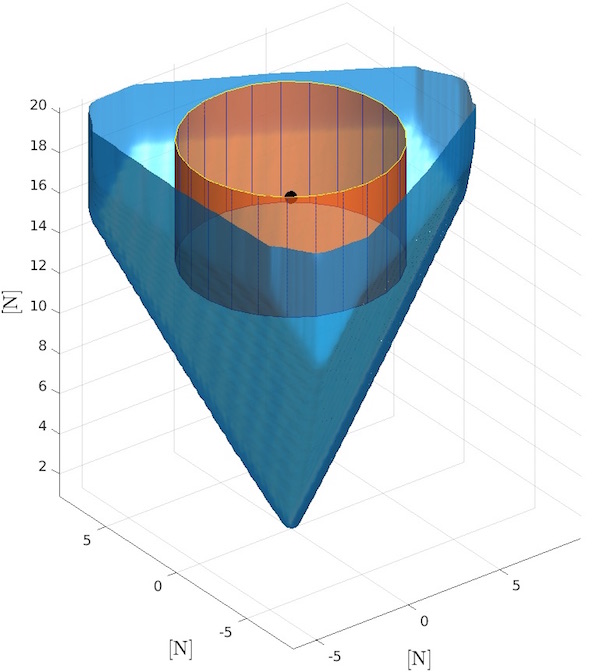}}
\def\figForceVolumeFull{\includegraphics[width=0.43\columnwidth]{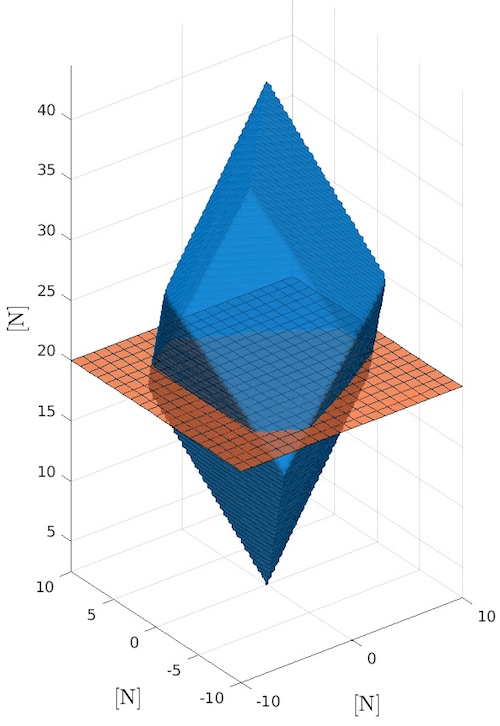}}

\def\figTiltHex{\includegraphics[width=1.\columnwidth]{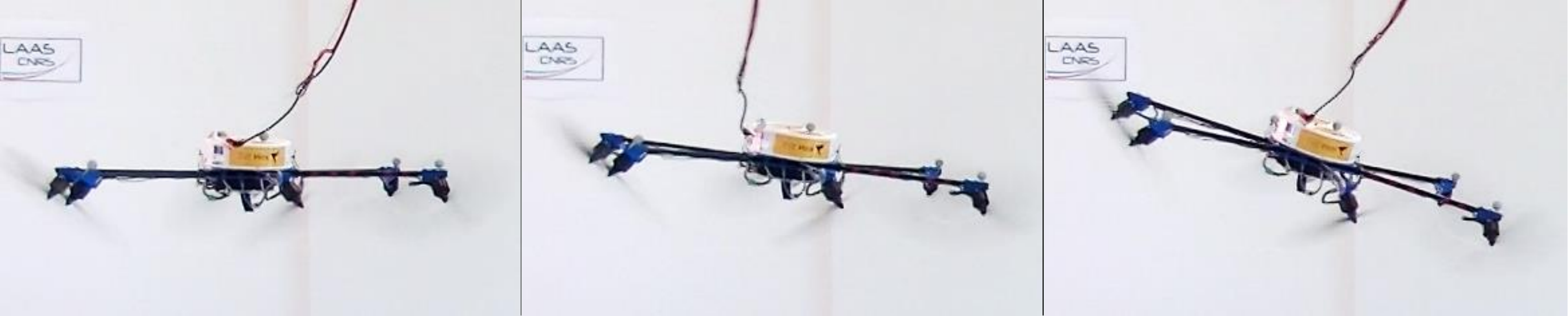}}

\title{Full-Pose Tracking Control  for Aerial Robotic Systems with Laterally-Bounded Input Force}

\author{Antonio Franchi$^1$, Ruggero Carli$^2$, Davide Bicego$^1$, and Markus Ryll$^1$
\thanks{$^1$LAAS-CNRS, Universit\'e de Toulouse, CNRS, Toulouse, France, \tt \footnotesize \href{mailto:afranchi@laas.fr}{afranchi@laas.fr},\href{mailto:dbicego@laas.fr}{dbicego@laas.fr},\href{mailto:mryll@laas.fr}{mryll@laas.fr}}
\thanks{$^2$Department of Information Engineering, University of Padova, Padova, Italy, \tt \footnotesize \href{mailto:carlirug@dei.unipd.it}{carlirug@dei.unipd.it}}
\thanks{This work has been funded by the European Union's Horizon 2020
    research and innovation programme under grant agreement No 644271 AEROARMS.}
}

\begin{document}

\maketitle

\begin{abstract}
In this paper, we define a general class of abstract aerial robotic systems named Laterally Bounded Force (LBF)  vehicles, in which most of the control authority is expressed along a principal thrust direction, while in the lateral directions a (smaller and possibly null) force may be exploited to achieve full-pose tracking. This class approximates well platforms endowed with non-coplanar/non-collinear rotors that can use the tilted propellers to slightly change the orientation of the total thrust w.r.t. the body frame.
For this broad class of systems, we introduce a new geometric control strategy in SE(3) to achieve, whenever made possible by the force constraints, the independent tracking of position-plus-orientation trajectories.
The exponential tracking of a feasible full-pose reference trajectory is proven using a Lyapunov technique in SE(3).
The method can deal seamlessly with both \emph{under-} and  \emph{fully-actuated} LBF platforms.
The controller guarantees the tracking of at least the positional part in the case that an unfeasible full-pose reference trajectory is provided.
The paper provides several experimental tests clearly showing the practicability of the approach and the sharp improvement with respect to state of-the-art approaches.
\end{abstract}
\section{Introduction}\label{sec:intro}

In the last years we have seen exciting new developments and astonishing applications in the field of unmanned aerial vehicles (UAVs). Advancements in control, perception and actuation allow UAVs to perform very agile maneuvers \cite{2010-MelMicKum,2017-FalMueFaeSca}. UAVs perform complex missions solely or in swarms and heterogenous groups~\cite{kushleyev2013towards, sherpa}. Originally performed missions have been contactless as, e.g., environmental monitoring and search and rescue~\cite{tomic2012toward}. Nowadays UAVs are as well utilized as aerial robots to perform direct physical interaction -- UAVs grasp, transport and manipulate our environment as shown by several research projects on aerial manipulation~\cite{arcas,aeroarms,aeroworks}.

Standard multi-rotors (quadrotors, hexarotors, octorotors, etc.) have coplanar propellers generating forces that are all aligned to one direction in body frame, which makes them under-actuated systems.
As a consequence, these platforms can produce a total force only along that direction. In order to follow a generic position trajectory the total force direction in world frame is changed by rotating the whole vehicle. Maneuvers in which rotation and translation are completely independent are precluded to these platforms.
Presence of such an underactuation does not only limit the set of maneuvers that the aerial vehicle can carry out, but even deteriorates its potentiality to interact with the environment by rapidly exerting forces in an arbitrarily-chosen direction of the space while keeping a pre-specified orientation. 
This could be a serious problem in the case that, e.g., the platform has to move through a hostile and cluttered ambient or resist a wind gust while keeping a desired attitude.

New actuation strategies that can overcome the aforementioned issues and then allow a complete tracking in position and attitude have been explored.
The major solution has been to mount rotors in a tilted way such that the thrust of each propeller is not collinear anymore. With this solution, the direction of the total force can be changed by selecting the intensity of the force produced by each propeller. If the number of propellers is  at least six, and tilting directions do not generate a singular configuration, then direction and intensity of both the instantaneous total moment and the instantaneous total force acting on the whole system can be made controllable at will (inside the input envelope).
This solution has been applied in several real implementation and is becoming more and more popular in the aerial vehicles and robotics community. For example, results achieved in~\cite{2012-VoyJia}
show an improvement in resisting an opposing wrench, while the work done in~\cite{2015e-RajRylBueFra,2016-BreDan,2016-ParHerJonKimLee} shows that such solution allows to decouple the tracking of a desired position and orientation. 

These new kind of platforms (sometimes referred to as \emph{fully-actuated}) call for new methods to control them efficiently and to reliably cope with the added complexity of the platforms and of the larger set of tasks in which they can be employed, when compared to standard collinear multi-rotors.
To fill this gap, in this paper we propose a novel method for controlling fully actuated multi-rotor platforms while taking into account the most limiting input bounds they have to cope with, i.e., lateral input force. The proposed controller ensures, in nominal conditions, the tracking of a full-6D pose (position plus orientation) reference trajectory. If the reference orientation and the force needed to track the position trajectory do not comply with the platform constraints, the proposed strategy gives priority to the tracking of the positional part of the trajectory\footnote{This choice is supported by, e.g., the fact that  in typical applications a wrong position tracking is more likely to lead to an obstacle crash than a non perfect orientation tracking.} while also tracking the closest feasible orientation to the reference one.

The proposed method is based on a rigorous analysis in which a formal convergence proof is provided. Furthermore, an extensive  experimental validation is carried out to test the viability of the theory and learn instructive lessons from practice. Theory and practice clearly demonstrate that the proposed method outperforms state-of-the-art methods, in terms of both performances and stability. 

In order to attain generality, the proposed method is designed to work with a large variety of different platforms which include not only the fully-actuated but also the under-actuated (collinear propeller) case. In this way we both maximize the breadth of the impact of the proposed methodology to new full-actuated platforms and we also maintain back-compatibility with the standard platforms.
This last feature makes the proposed method also ideal to control  vectored-thrust vehicles that can transit from an under-actuated to a fully-actuated configuration while flying -- as. e.g., the one presented in~\cite{2016j-RylBicFra}. In fact, by using the proposed controller there is no need of switching between two different controllers for each configuration.

Additionally, we envision that the method proposed in this paper will find large application in the new emerging topic of aerial robotic physical interaction with fully-actuated platforms. A striking example of this application has been shown in~\cite{2017e-RylMusPieCatAntCacFra}, where an admittance control framework has been built around  the controller developed in this paper, and real experiments in contact with the environment are shown.

A preliminary version of the method presented in this paper has been proposed in~\cite{2016j-RylBicFra}. With respect to~\cite{2016j-RylBicFra} we provide here a strongly improved content: 1) a controller that works on a more  general model (in~\cite{2016j-RylBicFra} only the specific case of hexarotor with a particular orientation pattern is considered), 2) the theoretical proof of the convergence of the controller (in~\cite{2016j-RylBicFra} no proof is given), 3) the analytic  solution of the optimization problem in a relevant case (not given in~\cite{2016j-RylBicFra}), and 4) a set of experiments with a real multi-rotor vehicle that validate the practicability of the method in real scenarios and the improvements with respect to the state of-the-art (in~\cite{2016j-RylBicFra} only a limited set of simulations are presented).

The remaining part of the paper is structured as follows. 
Section~\ref{sec:LBF} presents motivations and the state of the art.
Section~\ref{sec:model} presents the generic model. The full-pose geometric control and prove of the asymptotically exponential tracking are presented in Sec.~\ref{sec:control}. In Sec.~\ref{sec:control_case} we present the full computation of the generic controller in a meaningful case, while results of several experiments are shown in Sec.~\ref{sec:exps}. Finally we conclude the paper and give an outline of further possible extensions in Sec.~\ref{sec:concl}.

\subsection{Motivations and State of the Art}\label{sec:LBF}
In order to use at best the available energy, common multi-rotor platforms are designed with all coplanar rotors. Therefore the direction of the input force applied to the platform center of mass (the total thrust) is also collinear with the spinning axes of the rotors. Being the direction of the total thrust constant in body frame, those platforms are  underactuated. For these platforms several controllers have been proposed in the literature like, e.g., controller based on the dynamic feedback linearization~\cite{2001-MisBenMsi}, cascaded/backstepping controllers~\cite{2002-HamMahLozOst,2005-BouSie}, and geometric controllers on $SE(3)$~\cite{2006-MahChaHam,2010-LeeLeoMcc} (see~\cite{2013-HuaHamMorSam} for a  review of these and other possible strategies).
 
In the recent years new concepts have been developed where the use of non-coplanar propellers~\cite{2007-RomSalSanLoz,2012-VoyJia,2015e-RajRylBueFra}  allows the orientation of the total thrust to deviate from its principal direction, if needed. These platforms present several manifest benefits which  have been already discussed in Sec.~\ref{sec:intro}.
However, in order to minimize the waste of energy caused by the appearance of internal forces, the maximum component of the total thrust along the lateral direction is typically kept (by design) much  lower than the maximum allowed component along the vertical one.
We call these kind of platforms aerial vehicles with \emph{laterally-bounded force} (LBF).

An  LBF platform possesses a \emph{principal} direction of thrust along which most of the thrust can be exerted. A certain amount of thrust (typically smaller) can be exerted along any non-principal (lateral) directions. This model includes the standard  underactuated multi-rotor vehicle where thrust is possible only along the principal direction, and the isotropically fully-actuated platforms where a large amount of total thrust in the lateral directions is applicable~\cite{2016-ParHerJonKimLee,2016-BreDan}.

If the LBF platform is  underactuated then it is not able to track a generic full-pose trajectory, i.e., with independent position and orientation in $SE(3)$. The rotation about any axis that is orthogonal to the principal fixed total thrust direction must follow the evolution over time of the position trajectory, according to the well-known differential flatness property~\cite{2001-MisBenMsi,2011-MelKum}. 
The rotation about the axis that is parallel to the total thrust is instead independent from the position trajectory, therefore an underactuated multi-rotor aerial platform can only track a 4D-pose trajectory (i.e., position plus one angle).

On he contrary, if the LBF platform is fully-actuated then some force can be exerted in the lateral direction thus allowing the tracking of a generic full-pose (6D) trajectory. However due to the bounded thrust along the lateral directions, it is not possible to track \emph{any} full-pose trajectory.
 The larger the bounds the higher the ability of the platform to track any trajectory, the lower the bounds the more the platform resembles an underactuated multi-rotor and thus it becomes almost unable to track a full-pose trajectory but only a 4D-pose one. 

The easiest approach to control fully-actuated platforms is the inverse dynamics approach. First, a \emph{control wrench} is computed in order to track the desired trajectory by cancelling the nonlinear dynamical effects and zero the position and orientation errors. Then the thrust inputs for each propeller are computed from the control wrench by simply inverting the control allocation matrix. This method has been first proposed in~\cite{2015e-RajRylBueFra} and then used also in~\cite{2016-BreDan}\footnote{In~\cite{2016-BreDan} pseudo inversion is used in place of inversion in order to allocate the eighth control inputs of the octo-rotor to attain the six-dimensional wrench.}  and in~\cite{2016-ParHerJonKimLee}.
The limitation of this control approach is to neglect input saturation, which may easily lead to an unstable behavior if, e.g., the full-pose trajectory to be followed is not input-feasible.
Another control approach is presented in~\cite{2007-RomSalSanLoz}, which is however specific to that octorotor platform, it does not consider input bounds either, and is based on a particular Euler angle representation.

In this paper we present a geometric controller that is instead very general and applicable to any LBF vehicle, thus also taking into account the  bounds on the lateral control force.  The method is not prone to local orientation representation singularities since it is natively designed in $SE(3)$.
Furthermore, being not based on pure model inversion like feedback linearization, it is structurally more robust to model uncertainties.

\begin{remark}
It is worth to mention that another way to obtain orientation/position decoupling, could be to use less than six propellers plus additional servomotors to tilt the propellers while flying, as done, e.g., in~\cite{2015-RylBueRob}. This technology is typically referred to as \emph{vectoring thrust}. For example, in~\cite{2015-HuaHamMorSam} and~\cite{2015-RylBueRob} controllers that deal with vectoring thrust aerial vehicles are presented.
One main disadvantage of vectoring thrust platforms is the added complexity and weight due to the additional vectoring motor and mechanism. Another disadvantage is the much slower ability to change the force direction which is due to the small  maximum torque of lightweight servomotors and to the inertia of the motor-propeller group. The latter drawback, in particular, is an insurmountable obstacle for physical interaction application. Therefore we do not consider these platforms in the scope of this paper and we keep the  extension of the proposed controller to vectored thrust platforms as a possible future development.
\end{remark}

\section{Laterally-Bounded Force Aerial Vehicles}\label{sec:model}

With reference to Fig.~\ref{fig:model}, we denote with ${\cal F}_W=O_W,\{\mathbf{x}_W,\mathbf{y}_W,\mathbf{z}_W\}$ and ${\cal F}_B=O_B,\{\mathbf{x}_B,\mathbf{y}_B,\mathbf{z}_B\}$ the  fixed inertial frame, and body frame attached to the aerial platform, respectively. The origin of ${\cal F}_B$, i.e., $O_B$, is chosen coincident with the center of mass (CoM) of the platform and its position in ${\cal F}_W$ is denoted with ${\mathbf p}_{O_B}^W\in\mathbb{R}^3$, shortly indicated with just ${\mathbf p}$ in the following. For the reader's convenience, Table~\ref{tab:symbols} summarizes all the main symbols used in the paper.

\begin{figure}[t]
\centering
\includegraphics[width=1.0\columnwidth]{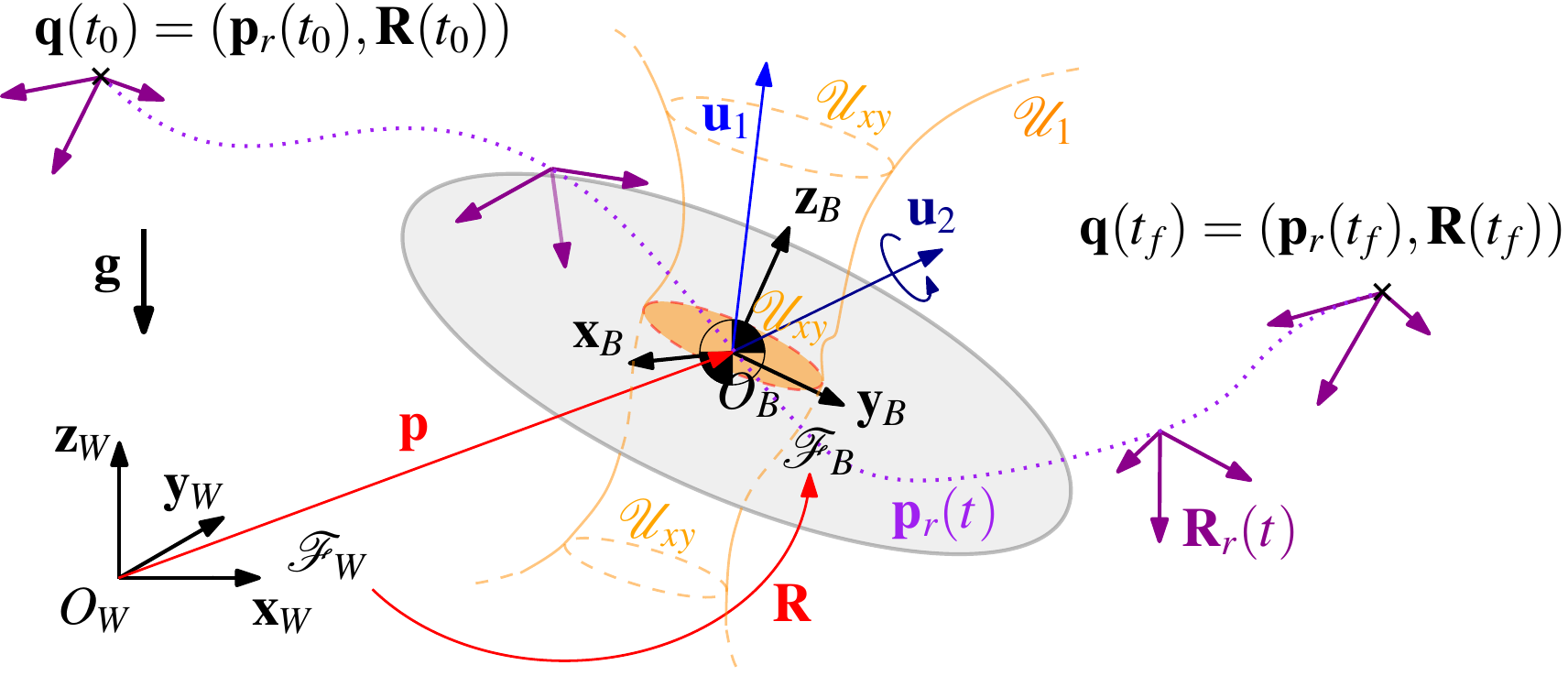}
\caption{A drawing illustrating the main the quantities of an LBF Aerial Vehicle, the main frames involved, the laterally bounded input sets and the full-pose 6D reference trajectory.}
\label{fig:model}
\end{figure}

The aerial platform is modeled as a rigid body whose mass is denoted with $m>0$. The positive definite matrix $\mathbf{J}\in\mathbb{R}^{3\times3}$ denotes the vehicle inertia matrix with respect to $O_B$ expressed in ${\cal F}_B$.
The orientation  of ${\cal F}_B$ with respect to ${\cal F}_W$ is represented by the rotation matrix $\mathbf{R}_B^W\in SE(3)$, shortly denoted with $\mathbf{R}$ in the following. 
The configuration of the aerial vehicle is defined by the position $\mathbf{p}$ and the attitude $\mathbf{R}$, which are gathered in the vehicle configuration $\mathbf{q}=(\mathbf{p},\mathbf{R})\in SE(3)$.

The angular velocity of ${\cal F}_B$ with respect to ${\cal F}_W$, expressed in ${\cal F}_B$, is indicated with ${\bm \omega}_{B,F}^B\in\mathbb{R}^3$, and briefly denoted as $\bm \omega$ in the following.
Thus, the vehicle orientation kinematics is described by:
\begin{align}
\dot{\mathbf{R}} &= \mathbf{R}[{\bm \omega}]_{\times},\label{eq:rot_kinem}
\end{align}
where $[\star]_{\times}:\mathbb{R}^3\to so(3)$ is the map that associates a vector $\star\in\mathbb{R}^3$ to its corresponding skew symmetric matrix.

Let us denote the control inputs of the vehicle with $\mathbf{u}_1 =
[u_1\; u_2 \; u_3]^T\in\mathbb{R}^3$ and $\mathbf{u}_2 =[u_4\; u_5\; u_6]^T\in\mathbb{R}^3$, representing the total force and  total moment  applied to the CoM of the vehicle expressed in ${\cal F}_B$, respectively.
The total force input $\mathbf{u}_1$ is subject to the following constraints 
\begin{align}
[u_1\;u_2]^T &\in {\cal U}_{xy}\subset \mathbb{R}^2,  \label{eq:input_constr_xy}\\
u_3 &\geq 0,  \label{eq:input_constr_z}
\end{align}
where the \emph{laterally bounding} set ${\cal U}_{xy}$ is a set that contains the origin. We define ${\cal U}_1= \{\mathbf{u}_1 \in \mathbb{R}^3\;|\; [u_1\;u_2]^T \in {\cal U}_{xy}, u_3 \geq 0\}$. Note that ${\cal U}_{xy}$ can be constant or even be changing depending  of $u_3$, as shown in Figures~\ref{fig:model} and~\ref{fig:conic_cylindric_inputs}--(left).

Using the Newton-Euler approach we can express the dynamics of the aerial platform as
\begin{align}
m\ddot {\mathbf p} &= -mg \mathbf{e}_3 + \mathbf{R}\mathbf{u}_1
    \label{eq:model_trans}\\
\mathbf{J}\dot	{\bm\omega} &= -{\bm\omega}\times \mathbf{J}{\bm\omega} + \mathbf{u}_2 \label{eq:model_rot}
\end{align}
with $g$ being the gravitational acceleration and $\mathbf{e}_i$, $i=1,2,3$, representing the $i$-th vector of the canonical basis of $\mathbb{R}_3$.

\begin{table}[t]
\caption{Main Symbols used in the paper}
\label{tab:symbols}
\centering
\renewcommand\arraystretch{1.0}
\resizebox{\columnwidth}{!}{
\begin{tabular}{ll}
	\toprule
	\textbf{Definition} & \textbf{Symbol} \\
	\midrule
	World Inertial Frame & ${\cal F}_W$ \\
	Attached Body Frame & ${\cal F}_B$ \\
	Position of the $O_B$, the CoM, in ${\cal F}_W$   & $\mathbf{p}$ \\
	Rotation matrix from ${\cal F}_W$ to ${\cal F}_B$   & $\mathbf{R}$ \\
	Angular velocity of ${\cal F}_B$ w.r.t ${\cal F}_W$ expr. in ${\cal F}_B$  & $\bm{\omega}$ \\
	Mass of the vehicle & $m$  \\
	Vehicle's Inertia matrix w.r.t to $O_B$ expressed in ${\cal F}_B$ & $\mathbf{J}$ \\
	Control force applied to the CoM expressed in ${\cal F}_B$  & $\mathbf{u}_1$ \\
	Control moment applied to the CoM expressed in ${\cal F}_B$  & $\mathbf{u}_2$\\
	Feasible set of the control force $\mathbf{u}_1$  & ${\cal U}_1$  \\
	Feasible set of the projection of $\mathbf{u}_1$ on the $xy$ plane in ${\cal F}_B$  &  ${\cal U}_{xy}$  \\
	\midrule
	Reference position for $O_B$ in ${\cal F}_W$   & $\mathbf{p}_r(t)$ \\
	Reference rotation matrix from ${\cal F}_W$ to ${\cal F}_B$ & $\mathbf{R}_r(t)$ \\
	\midrule 
	Reference control force to be applied to $O_B$   & $\mathbf{f}_r(t)$ \\
	Set of orientations in $SO(3)$ that allow the application of $\mathbf{f}_r(t)$   & ${\cal R}(\mathbf{f}_r)$\\
	Subset of ${\cal R}(\mathbf{f}_r)$ that minimizes a certain cost w.r.t. $\mathbf{R}_r$   & $\overline{\cal R}(\mathbf{f}_r,\mathbf{R}_r)$ \\
	Desired rotation matrix in $\overline{\cal R}(\mathbf{f}_r,\mathbf{R}_r)$    & $\mathbf{R}_d$ \\
	\bottomrule
\end{tabular}
}
\end{table}

\begin{remark}[Underactuated aerial vehicle] \it
When 
\begin{align}
{\cal U}_{xy}=\{\mathbf{0}\}\label{eq:U_xy_underactuated}
\end{align}
the total force is always oriented as $\mathbf{R}\mathbf{e}_3$ and model~\eqref{eq:model_trans}--\eqref{eq:model_rot} becomes the underactuated quadrotor model considered in~\cite{2001-MisBenMsi,2005-BouSie,2010-LeeLeoMcc}, see Fig.~\ref{fig:conic_cylindric_inputs}--(top). 
\end{remark}  

\smallskip  

\begin{remark}[Conic LBF] \it  
When 
\begin{align}
  {\cal U}_{xy}=\{[u_1\; u_2]^T\in \mathbb{R}^2 \; |\; u_1^2+u_2^2 \leq (\tan\alpha)^2u_3^2 \},
\end{align}
model~\eqref{eq:model_trans}--\eqref{eq:model_rot} approximates the case of hexarotors with tilted propellers~\cite{2012-VoyJia,2015e-RajRylBueFra,2016j-RylBicFra}, for which the set of allowable ${\cal U}_1$ forces has the conic shape depicted in Fig.~\ref{fig:conic_cylindric_inputs}-(middle).
The quantity $\alpha$ is a parameter that represents the tilting angle of the propellers (hexarotor).
\end{remark}

\begin{figure}[t]
\centering
\includegraphics[width=0.99\columnwidth]{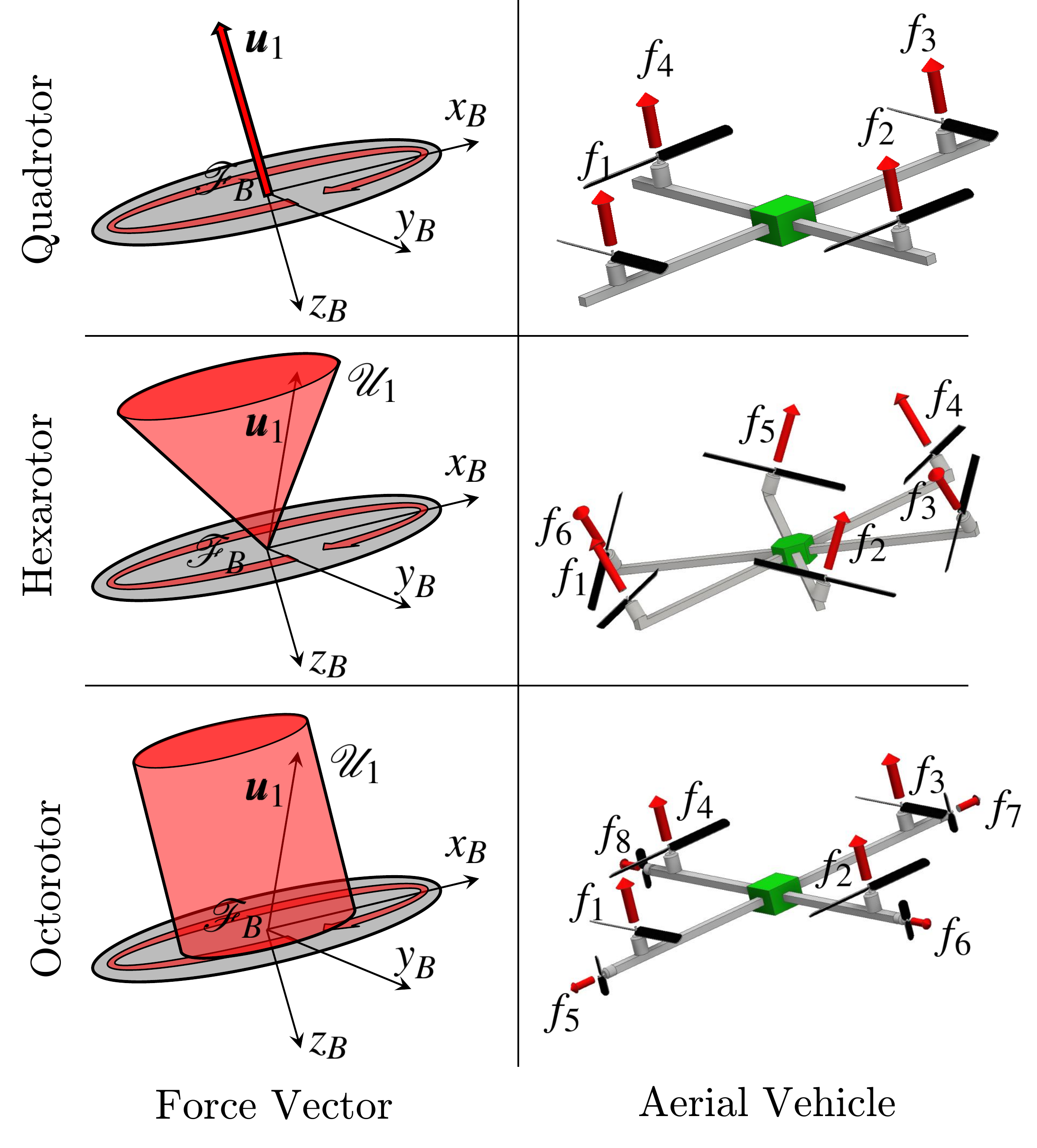}
\caption{Set of allowable total total thrust ${\cal U}_1$ for 3 basic different cases of LBF platforms. Top: a standard quadrotor. Middle: approximation of a hexarotor with tilted propellers. Bottom: approximation of an octorotor with smaller lateral propellers. In the last two cases the sets are an approximation of the real cases, which can be used to simplify the computations. However the proposed control can deal with the real more complex sets as well.}
\label{fig:conic_cylindric_inputs}
\end{figure}

\begin{remark}[Cylindric LBF] \it  
When 
\begin{align}
	\label{eq:r_xy}
  {\cal U}_{xy}=\{[u_1\; u_2]^T\in \mathbb{R}^2 \; |\; u_1^2+u_2^2 \leq r_{xy}^2 \},
\end{align}
model~\eqref{eq:model_trans}--\eqref{eq:model_rot} approximates the case of an octorotor (or, more in general, an $n$-rotor) with four main propellers and four (or $n-4$) secondary less powerful propellers tilted $90$~degrees w.r.t. the main ones, like the one presented in~\cite{2007-RomSalSanLoz}, for which the set of allowable ${\cal U}_1$ forces can be approximated by the cylindric shape depicted in Fig.~\ref{fig:conic_cylindric_inputs}-(bottom). Note that a more accurate (e.g., parallelepipedal) set  can be used if needed, for which the proposed controller is still valid.
The constant parameter $r_{xy}$ represents the maximum lateral thrust allowed by the smaller horizontal propellers. 
\end{remark}

It is worth to stress again that our model is general enough to encompass a much broader set of possible platform than the set given by just considering the previous examples.

\section{Full-Pose Geometric Control on SE(3)}\label{sec:control}

Let be given a full-pose trajectory $\mathbf{q}_r(t)=(\mathbf{p}_r(t),\mathbf{R}_r(t)): [t_0,t_f] \to SE(3)$, where $\mathbf{p}_r(t)\in\mathbb{R}^3$ is the reference position  trajectory and $\mathbf{R}_r(t)\in SO(3)$  is the reference attitude trajectory.
The nominal input required to track $\mathbf{q}_r(t)$ is easily obtained inverting~\eqref{eq:model_trans}-\eqref{eq:model_rot}
\begin{align} 
\mathbf{u}_{1}^r  &= \mathbf{R}_r^T\left(mg \mathbf{e}_3 + m\ddot {\mathbf p}_r\right)
    \label{eq:inver_dyn_u1}\\
\mathbf{u}_2^r  &= {\bm\omega}_r\times \mathbf{J}{\bm\omega}_r + \mathbf{J}\dot	{\bm\omega}_r,\label{eq:inver_dyn_u2}
\end{align}
where $\ddot{\mathbf p}_r=\frac{d^2}{dt^2}{\mathbf p}_r$, and ${\bm\omega}_r$ is defined using~\eqref{eq:rot_kinem}, i.e., $[{\bm \omega}_r]_{\times} = \mathbf{R}_r^T\dot{\mathbf{R}}_r$. 
\begin{dfn}
A reference trajectory $\mathbf{q}_r(t)$ is \emph{feasible} if $\mathbf{u}_{1}^r(t)\in {\cal U}_1$ $\forall t\in [t_0,t_f]$. 
\end{dfn}
Full-pose 6D tracking is possible only if the given reference trajectory $\mathbf{q}_r(t)$ is feasible. However we shall design a controller that works even if  $\mathbf{q}_r(t)$ is not feasible, in the sense that the tracking of $\mathbf{p}_r(t)$ is still guaranteed and no singularity appears. 
Consider the following position and velocity errors
\begin{align} 
\mathbf{e}_p &= \mathbf{p}-\mathbf{p}_r\\
\mathbf{e}_v &= \dot{\mathbf{p}}-\dot{\mathbf{p}}_r,
\end{align}
and the following vector 
\begin{align}
\mathbf{f}_r = m\ddot {\mathbf p}_r +  mg \mathbf{e}_3 - \mathbf{K}_p \mathbf{e}_p - \mathbf{K}_v \mathbf{e}_v,
\end{align}
where $\mathbf{K}_p$ and $\mathbf{K}_v$ are positive definite gain matrices. The vector $\mathbf{f}_r$  represents the reference total control force that ideally one would like to apply to the aerial vehicle CoM if the system  would be completely fully actuated, i.e., if ${\cal U}_1=\mathbb{R}^3$.

The  set of orientations that allow to apply $\mathbf{f}_r$ to the CoM of the LBF aerial vehicle is defined as
\begin{align}
{\cal R}(\mathbf{f}_r)  = \{\mathbf R\in SO(3) \;|\;  \mathbf{R}^T\mathbf{f}_r \in {\cal U}_1\}. 
\end{align}

For an underactuated coplanar multi-rotor system, i.e., if~\eqref{eq:U_xy_underactuated} holds, the set ${\cal R}(\mathbf{f}_r)$ is formed by any $\mathbf{R}$ such that $\mathbf{R}\mathbf{e}_3$ and $\mathbf{f}_r$ are parallel, i.e., $\mathbf{R}\mathbf{e}_3 \times \mathbf{f}_r =0$. For a generic LBF aerial vehicle the set ${\cal R}(\mathbf{f}_r)$ may contain also matrixes for which $\mathbf{R}\mathbf{e}_3 \times \mathbf{f}_r \neq 0$. Therefore we have the following. 
\begin{prop}
For any $\mathbf{f}_r$ it holds ${\cal R}(\mathbf{f}_r)\neq \emptyset$.
\end{prop}
\begin{proof}
If $\mathbf{f}_r\neq\mathbf{0}$ then, by definition of ${\cal R}(\mathbf{f}_r)$ and ${\cal U}_1$,  
\[
{\cal R}(\mathbf{f}_r) \supset \left\{\mathbf{R}\in SO(3) \,|\, \mathbf{R}\mathbf{e}_3 =\dfrac{\mathbf{f}_r}{\|\mathbf{f}_r\|}\right\} \neq \emptyset  
\]
If $\mathbf{f}_r = \mathbf{0}$ then ${\cal R}(\mathbf{f}_r)= SO(3)$. 
\end{proof}

As one can see from the LBF model the rotational dynamics~\eqref{eq:model_rot} is (completely) fully actuated and decoupled from the translational dynamics~\eqref{eq:model_trans}.
One of the main ideas behind the proposed controller is to exploit a cascaded structure by choosing, at each time $t$, a desired orientation $\mathbf{R}_d\in SO(3)$ that belongs to ${\cal R}(\mathbf{f}_r)$ and also minimizes a given cost function w.r.t.  $\mathbf{R}_r$. Then one can use the fully actuated rotational dynamics to track $\mathbf{R}_d$ and, in turns, track the reference position $\textbf{p}_r$. If the full-pose reference trajectory $\mathbf{q}_r$ is feasible then $\mathbf{R}_d$ will exponentially converge to $\mathbf{R}_r$ and then also the reference orientation will be tracked. Otherwise, only the best feasible orientation will be tracked. Therefore the controller implicitly prioritizes the position trajectory against the orientation one, as wanted. In the following, we shall formally define this controller concept and theoretically prove its convergence properties.
\smallskip

Define $\overline{\cal R}(\mathbf{f}_r,\mathbf{R}_r)\subset{\cal R}(\mathbf{f}_r)$ as the set of rotation matrices that solve the minimization
\begin{align}\label{eq:cost_function}
\min_{\mathbf{R}'\in{\cal R}(\mathbf{f}_r)} J(\mathbf{R}_r,\mathbf{R}'),
\end{align}
where $J:SO(3)\times SO(3) \to \mathbb{R}_{\geq 0}$ is a suitably chosen cost function that represents the degree of similarity between $\mathbf{R}_r$ and $\mathbf{R}'$ one is interested in. The elements in $\overline{\cal R}(\mathbf{f}_r,\mathbf{R}_r)$ represent orientations of the LBF that allow to apply $\mathbf{f}_r$ and minimize the function $J$ w.r.t.  $\mathbf{R}_r$.

Consider that, at each time $t$ a desired orientation $\mathbf{R}_d\in \overline{\cal R}(\mathbf{f}_r,\mathbf{R}_r)$ is chosen. Furthermore, whenever $\mathbf{R}_r\in\overline{\cal R}(\mathbf{f}_r,\mathbf{R}_r)$ then $\mathbf{R}_d$ must bet chosen equal to $\mathbf{R}_r$. 

Then define, as in~\cite{2010-LeeLeoMcc}, the rotation and angular velocity errors
\begin{align} 
\mathbf{e}_R &= \frac{1}{2}(\mathbf{R}_d^T\mathbf{R}-\mathbf{R}^T\mathbf{R}_d)^\vee,\label{eq:rot_error}\\
\mathbf{e}_\omega &= {\bm\omega}-\mathbf{R}\mathbf{R}_d^T{\bm\omega}_d\label{eq:ang_vel_error}
\end{align}
where $\bullet^\vee: so(3)\to \mathbb{R}^3$ is the inverse map of $[\star]_\times$, and ${\bm\omega}_d$ is the angular velocity associated to $\mathbf{R}_d$.
Consider then the following control law
\begin{align}
\mathbf{u}_1 &=  
{\rm sat}_{{\cal U}_{xy}}
\big(
(\mathbf{f}_r^T
\mathbf{R} \mathbf{e}_1) \mathbf{e}_1
+
(\mathbf{f}_r^T\mathbf{R} \mathbf{e}_2) \mathbf{e}_2
\big)
 +(\mathbf{f}_r^T\mathbf{R} \mathbf{e}_3) \mathbf{e}_3\label{eq:u1_ctrl}
\\
\mathbf{u}_2 &= {\bm\omega}\times \mathbf{J}{\bm\omega} - \mathbf{K}_R \mathbf{e}_R - \mathbf{K}_\omega \mathbf{e}_\omega +  \label{eq:u2_ctrl}\\
&~~~~~~~~~~~~~~~~~~ -\mathbf{J}
\big(
[\bm{\omega}]_\times \mathbf{R}^T \mathbf{R}_d \bm{\omega}_d
- \mathbf{R}^T \mathbf{R}_d \dot{\bm{\omega}}_d
\big)
\notag
\end{align}
where ${\rm sat}_{{\cal U}_{xy}}(\mathbf{x})$ is a vector in ${\cal U}_{xy}$ with the same direction of $\mathbf{x}$, that minimizes the distance from $\mathbf{x}$. 
$\mathbf{K}_R=k_R\textbf{I}$ and $\mathbf{K}_\omega=k_\omega\textbf{I}$ are the gain matrices with $k_R>0$ and $k_\omega>0$.%

A nice property is that~\eqref{eq:u1_ctrl}--\eqref{eq:u2_ctrl} reduces to the control in~\cite{2010-LeeLeoMcc} in the case that~\eqref{eq:U_xy_underactuated} holds (underactuated vehicles). However a major difference is in the computation of $\mathbf{R}_d$, which in~\cite{2010-LeeLeoMcc} is computed from the position trajectory using the differential flatness property while in the proposed controller is computed ensuring feasibility of the input and minimizing the cost function w.r.t. the reference orientation $\textbf{R}_r$.

In order to prove the convergence properties of the proposed controller let us consider the following error function between two rotation matrixes $\mathbf{R}_1$ and $\mathbf{R}_2$ to be
\begin{align}
d(\mathbf{R}_1,\mathbf{R}_2) = \frac{1}{2}{\rm tr}\left(\mathbf{I} - \mathbf{R}_2^T\mathbf{R}_1\right).
\end{align}

\begin{thm}\label{thm:convergence_p_r_R_d}
Assume that 
${\bf R}_d(t) \in {\cal R}(\mathbf{f}_r(t))$ for any $t$ and that ${\bm \omega}_d(t)$ and $\dot{{\bm \omega}}_d(t)$ are well defined for any $t$. 
Consider the control ${\bf u}_1$ and ${\bf u}_2$ defined at \eqref{eq:u1_ctrl} and \eqref{eq:u2_ctrl}.

Assume that the initial condition satisfies
\begin{equation}\label{eq:InCondR}
d\left({\bf R}(0), {\bf R}_d(0)\right) < 2,
\end{equation}
and
\begin{equation}\label{eq:InCondErrOmega}
\|{\bf e}_{\omega}(0)\|^2 < \frac{2}{\lambda_{\min}(J)} k_R \, \left(\,1-d\left({\bf R}(0), {\bf R}_d(0)\right)\,\right)
\end{equation}
Then, the zero equilibrium of the tracking errors ${\bf e}_R$, ${\bf e}_{\omega}$, ${\bf e}_p$ and ${\bf e}_v$ is exponentially stable. The region of attraction is characterized by \eqref{eq:InCondR} and \eqref{eq:InCondErrOmega}.
\end{thm}

\begin{proof}
The proof is divided into two parts. We first show that, if the $\mathbf{R}(0)$ and ${\bf e}_{\Omega}(0)$ satisfy, respectively, \eqref{eq:InCondR} and \eqref{eq:InCondErrOmega}, then $\mathbf{R}(t)$ converges exponentially to $\mathbf{R}_d(t)$, in the sense that the function $d\left({\bf R}(t), {\bf R}_d(t)\right)$ goes exponentially to zero. Secondly, we characterize the error dynamics on the translational dynamics and, 
based on the fact that $\mathbf{R}(t)$ converges exponentially to $\mathbf{R}_d(t)$, we show that also ${\bf e}_p$ and ${\bf e}_v$ goes exponentially to zero.\\
The time derivative of ${\bf e}_\omega$ defined in~\eqref{eq:ang_vel_error} is 
\begin{align}
\mathbf{J} \dot{{\bf e}}_\omega= \mathbf{J} \dot{\bm \omega}+\mathbf{J} \left(\left[{\bm \omega}\right]_\times {\bf R}^T {\bf R}_d {\bm \omega}_d -{\bf R}^T {\bf R}_d \dot{{\bm \omega}}_d \right).
\label{eq:w_err_dynam}
\end{align}
Plugging~\eqref{eq:model_rot} into~\eqref{eq:w_err_dynam} and substituting  ${\bf u}_2$ from~\eqref{eq:u2_ctrl},
we get
$$
\mathbf{J} \dot{{\bf e}}_\omega =-k_R {\bf e}_R-k_\omega {\bf e}_\omega.
$$
In~\cite{2010-LeeLeoMcc}, it is shown, by exhibiting a suitable Lyapunov function, that, under conditions in \eqref{eq:InCondR} and in \eqref{eq:InCondErrOmega}, the zero equilibrium of the attitude tracking error ${\bf e}_R$, ${\bf e}_\omega$ is exponentially stable and that there exist two positive constants $\alpha, \beta$ such that 
\begin{equation}\label{eq:dist_R_Rd}
d\left({\bf R}(t), {\bf R}_d(t)\right) < \alpha e^{-\beta t} d\left({\bf R}(0), {\bf R}_d(0)\right).
\end{equation}

We determine now the error dynamics of the translational dynamics. 
Substituting $\mathbf{u}_1$ from~\eqref{eq:u1_ctrl} in~\eqref{eq:model_trans} we obtain
\begin{align*}
m\ddot {\mathbf p} & = -mg \mathbf{e}_3 + {\bf f}_r+ \mathbf{R}\mathbf{u}_1-  {\bf f}_r
 = m\ddot {\mathbf p}_r   - \mathbf{K}_p \mathbf{e}_p - \mathbf{K}_v \mathbf{e}_v + \bm \gamma,
\end{align*}
where $\gamma= \mathbf{R}\mathbf{u}_1-  {\bf f}_r$. It easily follows that
\begin{align}
m\dot {\mathbf e}_v &=  - \mathbf{K}_p \mathbf{e}_p - \mathbf{K}_v \mathbf{e}_v + \bm \gamma.
\label{eq:ErrorTrasDyn}
\end{align}
Consider \eqref{eq:dist_R_Rd} and observe that, since ${\bf R}_d \in {\cal R}(\mathbf{f}_r)$ for any $t$, we have that there exist two positive constants $C$, $\rho$ such that
$$
\|\bm \gamma(t) \| \leq C e^{-\rho t} \|\bm \gamma(0)\|.
$$
Let ${\bf x}= \left[{\bf e}_v \,\,\,\,{\bf e}_p  \right]^T$ then, \eqref{eq:ErrorTrasDyn} can be written in vector form as
\begin{equation}\label{eq:ErrorTraDynVector}
\dot{{\bf x}}={\bf A} {\bf x}+ {\bf B} {\bm \gamma}
\end{equation}
where 
$$
{\bf A}=\left[
\begin{array}{cc}
- \mathbf{K}_v  & - \mathbf{K}_p \\
{\bf I} & {\bf 0} 
\end{array}
\right], \qquad {\bf B}=\left[
\begin{array}{c}
{\bf I} \\
{\bf 0} 
\end{array}
\right].
$$
Since $\mathbf{K}_v, \mathbf{K}_p$ are both positive definite matrices, we have that ${\bf A}$ is a Hurwitz matrix.

Observe that \eqref{eq:ErrorTraDynVector} is the cascade of a linear stable system and an exponential stable signal. Then we can apply Lemma \ref{lem:cascade_system} (reported in Appendix) obtaining the result stated in the Theorem.
\end{proof}

\begin{figure}[t]
\centering
\includegraphics[width=\columnwidth]{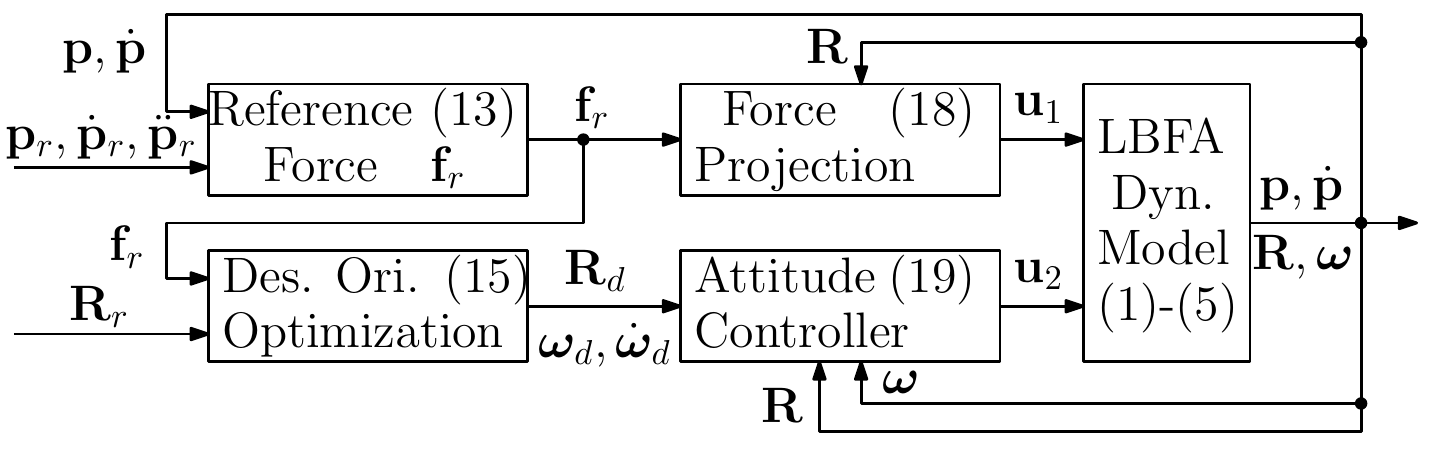}
\caption{Block diagram of the proposed geometric controller with the references to the corresponding equations in the text.}
\label{fig:block_ctrl}
\end{figure}

A block diagram that shows the main subsystems of the proposed control architecture is provided in Fig.~\ref{fig:block_ctrl}.

Theorem~\ref{thm:convergence_p_r_R_d} ensures, under mild conditions, the exponential stability of $\mathbf{e}_p$, $\mathbf{e}_v$, $\mathbf{e}_R$, and $\mathbf{e}_\omega$. Notice that this results holds regardless of the feasibility of $\mathbf{q}_r$. If $\mathbf{q}_r$ is also feasible then exponential tracking of $\mathbf{q}_r$ by $\mathbf{q}$ is also guaranteed. In order to formally state this fact let us define the following errors
\begin{align} 
\mathbf{e}_{R_r} &= \frac{1}{2}(\mathbf{R}_r^T\mathbf{R}_d-\mathbf{R}_d^T\mathbf{R}_r)^\vee,\label{eq:rot_ref_error}\\
\mathbf{e}_{\omega_r} &= {\bm\omega}_d-\mathbf{R}_d\mathbf{R}_r^T{\bm\omega}_r.\label{eq:ang_vel_ref_error}
\end{align}

In next result we characterize the convergence of the above errors to zero provided that the reference trajectory $\mathbf{q}_r(t)$ is \emph{feasible} and satisfies the additional property that $\mathbf{u}_{1}^r$ is \emph{sufficiently inside} ${\cal U}_1$, meaning that there exists a time instant $\bar t$ and a positive number $\epsilon$ such that the distance of $\mathbf{u}_{1}^r$ from the boundary of ${\cal U}_1$ is greater than $\epsilon>0$ for all $t > \bar{t}$, i,e,
\begin{equation}\label{eq:AddProp}
\text{dist}\left(\mathbf{u}_{1}^r(t), \partial {\cal U}_1\right)>\epsilon , \qquad \forall \, t> \bar{t}.
\end{equation}

\begin{thm}\label{prop:convergence_R_d_R_r}
Assume $\mathbf{q}_r(t)$ is a feasible trajectory and that it satisfies the additional property in \eqref{eq:AddProp}. 
Assume that 
${\bf R}_d(t) \in {\cal R}(\mathbf{f}_r(t))$ for any $t$ and that ${\bm \omega}_d(t)$ and $\dot{{\bm \omega}}_d(t)$ are well defined for any $t$. Consider the control ${\bf u}_1$ and ${\bf u}_2$ defined at \eqref{eq:u1_ctrl} and \eqref{eq:u2_ctrl}. 
Assume that the initial condition satisfies \eqref{eq:InCondR} and \eqref{eq:InCondErrOmega}.
Then the zero equilibrium of the tracking errors ${\bf e}_R$, ${\bf e}_{\omega}$, ${\bf e}_p$ and ${\bf e}_v$ is exponentially stable and there exists a time instant $\bar{t}\geq t_0$ such that $\mathbf{e}_{R_r}(t)=\mathbf{e}_{\omega_r}(t)=0$ for all $t >\bar{t}$. 
The region of attraction is characterized by \eqref{eq:InCondR} and \eqref{eq:InCondErrOmega}.
\end{thm}

\begin{proof}
From Theorem~\ref{thm:convergence_p_r_R_d} we can write that
\begin{align}
\mathbf{f}_r = m\ddot {\mathbf p}_r +  mg \mathbf{e}_3 + \bm \xi
\end{align}
where 
\begin{equation}\label{eq:PropertyXi}
\|\bm\xi(t)\| \leq L e^{-\lambda t}\|\bm\xi(0)\|
\end{equation} 
for some positive constants $L$ and $\lambda$. This implies that the vector $\mathbf{f}_r- (m\ddot {\mathbf p}_r +mg \mathbf{e}_3)$ and, in turn, also the vector 
${\bf R}_r^T\mathbf{f}_r- {\bf R}_r^T(m\ddot {\mathbf p}_r + mg \mathbf{e}_3)$, tend exponentially to zero.

Hence, since $\mathbf{q}_r(t)$ is feasible and satisfies \eqref{eq:AddProp}, it follows, from continuity arguments, that there exists $t'$such that $\mathbf{R}_r^T  \mathbf{f}_r  \in {\cal U}_1$ for all $t \geq t'$. Therefore ${\bf R}_d(t)= {\bf R}_r(t)$ for all $t > t'$.
\end{proof}

\begin{remark}
The proposed controller (in particular the attitude controller~\eqref{eq:u2_ctrl}) relies on the availability of $\bm{\omega}_d$, and $\dot{\bm{\omega}}_d$. These quantities depend in turn on $\mathbf{R}_d$ which is the output of an optimization algorithm executed at each control step. In order for $\bm{\omega}_d$ and $\dot{\bm{\omega}}_d$ to be well defined and available the optimization must ensure a sufficient smoothness of $\mathbf{R}_d$. This could be enforced by adding, e.g., a regularization term in the cost function $J$. 
If in the real case at hand this is not possible (or not implementable), then at each time instant in which $\mathbf{R}_d$ is not smooth the attitude controller will undergo a new transient phase. In practice, see Sec.~\ref{sec:exps}, we have experimentally ascertain that the presence of a few isolated non-smooth instants  does not constitute at all a real problem for the stability of the implementation and that regularization is actually not needed for practical stabilization.
\end{remark}

\section{Computation of $\mathbf{R}_d$ for an Important Case}\label{sec:control_case}

The proposed control method is kept on purpose general regarding two main features: the choice of ${\cal U}_{xy}$ in~\eqref{eq:input_constr_xy} and the choice of $J$ in~\eqref{eq:cost_function}. The former allows the method to be used for a large set of aerial vehicles with different actuation capabilities. The latter allows the engineer to customize the definition of similarity between two orientations in order to comply with the particular task at hand. In this section we illustrate how these two general features are particularized to a specific meaningful case.

In particular, we consider the case of ${\cal U}_{xy}$ defined in~\eqref{eq:r_xy} and the following choice of cost function $J$
\begin{align}\label{eq:cost_b3}
J(\mathbf{R}_r,\mathbf{R}') = 1 - \mathbf{b}^T_{3r}\mathbf{b}_3',
\end{align}
where 
$\mathbf{R}_r = [\mathbf{b}_{1r}\, \mathbf{b}_{2r}\, \mathbf{b}_{3r}]$ and $\mathbf{R}' = [\mathbf{b}_{1}'\, \mathbf{b}_{2}'\, \mathbf{b}_{3}']$. The cost function $J$ in~\eqref{eq:cost_b3} is minimized whenever $\mathbf{b}_{3r}=\mathbf{b}_{3}'$ and maximized whenever $\mathbf{b}_{3r}=-\mathbf{b}_{3}'$. 

\begin{remark} \it  
For example in the case that the aerial vehicle embeds a down-facing camera,  the cost~\eqref{eq:cost_b3} can be used to emphasize the fact that:
\begin{inparaenum}
\item while following the position trajectory it is important to orient the sagittal axis of the camera as prescribed by $\mathbf{b}_{3r}$, but 
\item the particular  rotation about the sagittal axis has no relevance for the specific camera-observation task.
\end{inparaenum}
\end{remark}

In the following we show how it is possible to efficiently compute an $\mathbf{R}_d$ that belongs to $\overline{\cal R}(\mathbf{f}_r,\mathbf{R}_r)$ and is also equal to $\mathbf{R}_r$ if $\mathbf{R}_r\in\overline{\cal R}(\mathbf{f}_r,\mathbf{R}_r)$. These are in fact the requirements needed for $\mathbf{R}_d$ in order for Theorems~\ref{thm:convergence_p_r_R_d} and~\ref{prop:convergence_R_d_R_r} to be valid.

Let us start by instantiating ${\cal R}(\mathbf{f}_r)$ for this particular case. From simple geometrical considerations on the cylindrical shape of the set ${\cal U}_1$  it is easy to see that the following definition of ${\cal R}(\mathbf{f}_r)$ holds
\begin{align}\label{eq:feasible_set_case}
{\cal R}(\mathbf{f}_r) = \left\{\mathbf R'\in SO(3) \;|\; \mathbf{f}^T_r\mathbf{b}_3' \geq \sqrt{\|\mathbf{f}_r\|^2 - r_{xy}^2 }\right\},
\end{align}
which states that the vector $\mathbf{f}_r$ must lie within the cylinder of radius $r_{xy}$ generated about the axis $\mathbf{b}_3'$ (see Fig.~\ref{fig:feasiblity_case}).
 
 \begin{figure}[t]
\centering
\includegraphics[width=0.42\columnwidth]{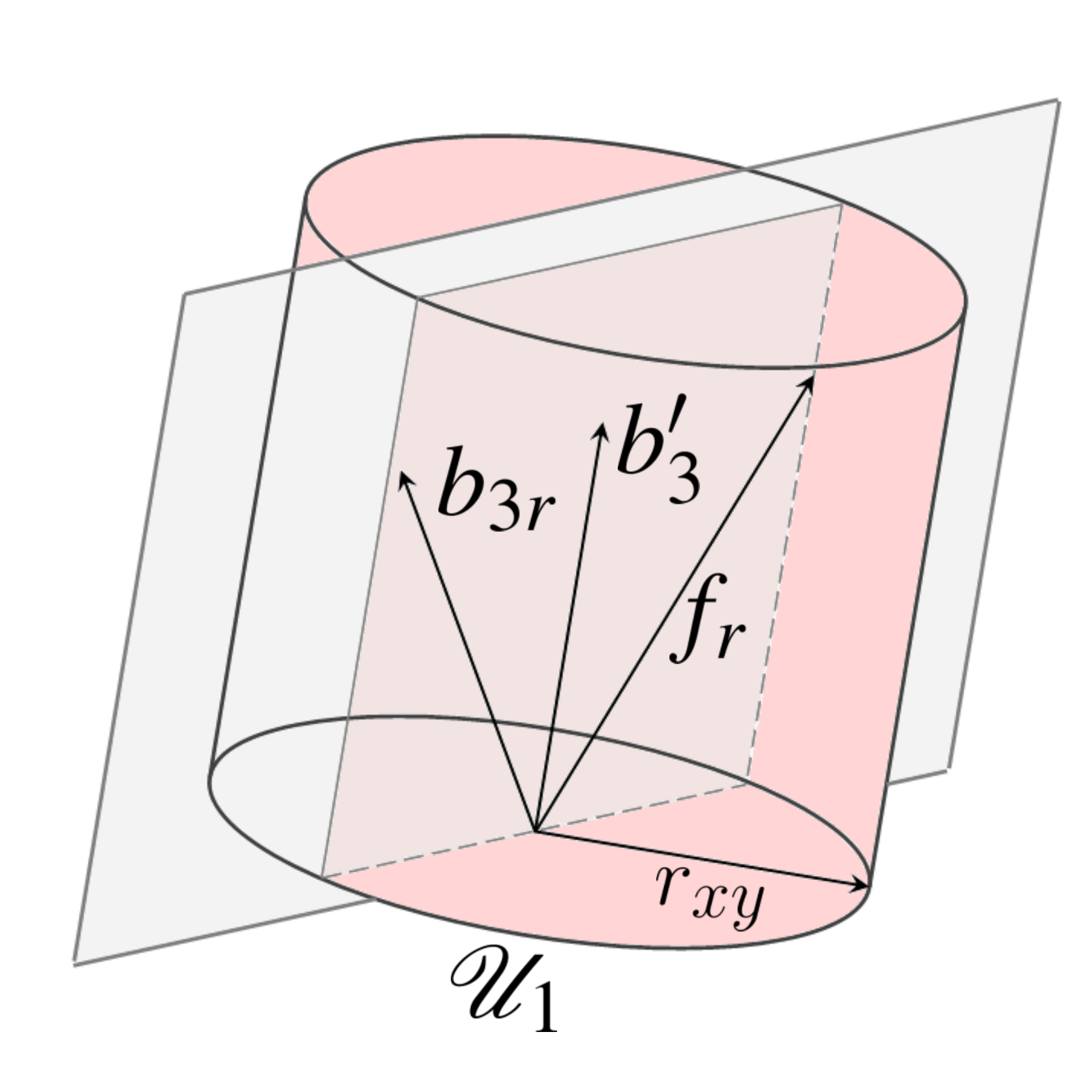}
\caption{Geometrical interpretation of the feasibility constraint in~\eqref{eq:feasible_set_case}.}
\label{fig:feasiblity_case}
\end{figure}

Using~\eqref{eq:cost_b3} and \eqref{eq:feasible_set_case} we can rewrite~\eqref{eq:cost_function} in terms of the only vector variable $\mathbf{b}_3'$, instead of the whole matrix $\mathbf{R}'$, as
\begin{align}\label{eq:cost_function_case}
\min_{
\substack{
\mathbf{f}^T_r\mathbf{b}_3' \geq \sqrt{\|\mathbf{f}_r\|^2 - r_{xy}^2 },\\
\|\mathbf{b}_3'\|^2 = 1
}
}
\; - \mathbf{b}^T_{3r}\mathbf{b}_3',
\end{align}
where $r_{xy}^2$, $\mathbf{f}^T_r$ and $\mathbf{b}^T_{3r}$ are the givens of the problem.

In the case that $\mathbf{f}^T_r\mathbf{b}_{3r} \geq \sqrt{\|\mathbf{f}_r\|^2 - r_{xy}^2 }$ then $\mathbf{b}_3'=\mathbf{b}_{3r}$ is the solution to~\eqref{eq:cost_function_case}. Otherwise, let us write $\mathbf{b}_3'$ as the sum of two components $\mathbf{b}_3'=\mathbf{b}_{3\parallel}' + \mathbf{b}_{3\perp}'$, where $\mathbf{b}_{3\parallel}'$ is parallel to the plane spanned by $\mathbf{b}_{3r}$
and $\mathbf{f}_r$, while
$\mathbf{b}_{3\perp}'$ is perpendicular to it, i.e., is parallel to $\mathbf{b}_{3r}\times \mathbf{f}_r$.
It is easy to see that the cost function in~\eqref{eq:cost_function_case}, i.e., $\mathbf{b}^T_{3r}\mathbf{b}_3'$, is not affected by $\mathbf{b}_{3\perp}'$, in fact 
\[
\mathbf{b}^T_{3r}\mathbf{b}_3' = \underbrace{\mathbf{b}^T_{3r}\mathbf{b}_{3\perp}'}_{=0} + \mathbf{b}^T_{3r}\mathbf{b}_{3\parallel}' = \mathbf{b}^T_{3r}\mathbf{b}_{3\parallel}'.
\]

The vector $\mathbf{b}_{3\parallel}'$ can be written using the Rodrigues's rotation formula as
\[
\mathbf{b}_{3\parallel}'(\theta) = \mathbf{b}_{3r} \cos\theta + (\mathbf{k} \times \mathbf{b}_{3r})\sin\theta + \mathbf{k} (\mathbf{k} \cdot \mathbf{b}_{3r}) (1 - \cos\theta),
\]
where $\mathbf{k}=\frac{\mathbf{b}_{3r}\times \mathbf{f}_r}{\|\mathbf{b}_{3r}\times \mathbf{f}_r\|}$ and $\theta$ is the rotation angle that univocally defines $\mathbf{b}_{3\parallel}'$. Noting that the constraint $\|\mathbf{b}_3'\|^2 = 1$ is automatically verified by $\mathbf{b}_{3\parallel}'(\theta)$ for any $\theta$, then we can further simplify~\eqref{eq:cost_function_case} in terms of the only scalar variable $\theta$ as
\begin{align}\label{eq:cost_function_case_theta}
\min_{
\mathbf{f}^T_r\;\mathbf{b}_{3\parallel}'(\theta) \geq \sqrt{\|\mathbf{f}_r\|^2 - r_{xy}^2 }
}
\; - \mathbf{b}^T_{3r}\;\mathbf{b}_{3\parallel}'(\theta).
\end{align}
The minimization prolem~\eqref{eq:cost_function_case_theta} can be efficiently solved numerically using the bisection method shown in Algorithm~\ref{algo:bisection}.

In order to finally compute $\mathbf{R}_d$ from $\mathbf{b}_{3d}$ we suggest to employ the following formula,  as done in~\cite{2010-LeeLeoMcc}:
\begin{align}\label{eq:computation_R_d_case}
\mathbf{R}_d = 
\Bigg[
\underbrace{(\mathbf{b}_{3d} \times \mathbf{b}_{1r})\times \mathbf{b}_{3d}}_{\mathbf{b}_{1d}} \quad \underbrace{\mathbf{b}_{3d} \times \mathbf{b}_{1r}}_{\mathbf{b}_{2d}} \quad \mathbf{b}_{3d}
\Bigg].
\end{align}

Finally, we note that if $\mathbf{R}_r\in\overline{\cal R}(\mathbf{f}_r,\mathbf{R}_r)$ then $\mathbf{f}^T_r\mathbf{b}_{3r} \geq \sqrt{\|\mathbf{f}_r\|^2 - r_{xy}^2 }$ which, as we previously said, implies that $\mathbf{b}_{3d}=\mathbf{b}_{3r}$. Then, from~\eqref{eq:computation_R_d_case} it results $\mathbf{R}_d=\mathbf{R}_r$, thus allowing to fulfill also the second requirement on the computation of $\mathbf{R}_d$, besides the minimization~\eqref{eq:cost_function}.

\begin{algorithm}[t]
\caption{{\it Bisection used to solve problem \eqref{eq:cost_function_case_theta}}}

\nllabel{algo:bisection}                  

\small{

\KwData{$n$  (number of iterations $\propto$ solution accuracy)}

\KwData{$\mathbf{b}_{3r}$, $\mathbf{f}_r$, and $r_{xy}$ (givens of the problem)}

$\mathbf{k} \leftarrow \frac{\mathbf{b}_{3r}\times \mathbf{f}_r}{\|\mathbf{b}_{3r}\times \mathbf{f}_r\|}$

$\theta_{max} \leftarrow  \arcsin(\|\mathbf{k}\|)$\;

$\theta \leftarrow  \theta_{max}/2$\; 

\For{$i=1$ to $n$}
{
	
	\eIf{$\quad\mathbf{f}^T_r\;\mathbf{b}_{3\parallel}'(\theta) \geq \sqrt{\|\mathbf{f}_r\|^2 - r_{xy}^2}\quad$}
	{
            $\theta \leftarrow \theta - \frac{1}{2}\frac{\theta_{max}}{2^i}$\;  
	}{
            $\theta \leftarrow \theta + \frac{1}{2}\frac{\theta_{max}}{2^i}$\; 
	}

}

\Return $\theta$

}

\end{algorithm}

It is worth to notice that the described algorithm takes a negligible time to be run on a standard computer, thus allowing a real time numerical control implementation at frequencies that are way below $1$\,ms for each control loop, if needed by the application.

In case of different sets ${\cal U}_{xy}$ and different cost functions $J$ either similar efficient approaches can be used or the method presented here can be used as a conservative approximation.

\section{Experiments}\label{sec:exps}

In this section we present and discuss the results of different experiments performed with an LBF aerial vehicle. 
The experimental setup used to conduct the experiments is introduced in Sec.~ \ref{sec:ExperimentalSetup} while in Sec.~\ref{sec:ExperimentalResults} we present and discuss thoroughly the different experiments.

\subsubsection{Experimental Setup}
\label{sec:ExperimentalSetup}
The LBF platform used to perform the experiments is the Tilt-Hex robot, an in-house developed fully actuated vehicle (see Fig.~\ref{fig:Tilt-Hex}). The Tilt-Hex is an example of platform that can be controlled by the full-pose geometric controller presented in Sec.~\ref{sec:control}. We  carefully chose the Tilt-Hex to perform the  experiments as the platform is able to emulate different values of $r_{xy}\in [0, r_{xy_{max}}]$ (see~\eqref{eq:r_xy}), which means that we can emulate seamlessly underactuated and fully actuated platforms. 
Notice that even if the Tilt-Hex is an hexarotor 
that implements the concept depicted in Fig.~\ref{fig:conic_cylindric_inputs}-(middle), we can emulate for validation purposes the  cylindric force constraint of Fig.~\ref{fig:conic_cylindric_inputs}-(bottom) by choosing $r_{xy}$ small enough and avoiding trajectories that require extreme values of $u_3$, see Fig.~\ref{fig:cylinder_in_cone}.

 All components of the Tilt-Hex are off-the-shelf available or 3D printable. The diameter of the Tilt-Hex (distance between two propeller hubs) is $\SI{0.8}{\meter}$. The total mass, including a \SI{2200}{\milli\ampere\hour} battery, is $m=\SI{1.8}{\kilo\gram}$. The propellers are tilted first about the axis that connects the rotor with the center of the hexarotor shape ($\alpha = \pm\SI{35}{\degree}$), then about the axis that is perpendicular to the previous rotation axis and $\vect{e}_3$ ($\beta = \SI{25}{\degree}$). These angles ensure a balanced selection between large lateral forces and inefficient losses as a result of internal forces.
 
  \begin{figure}[t]
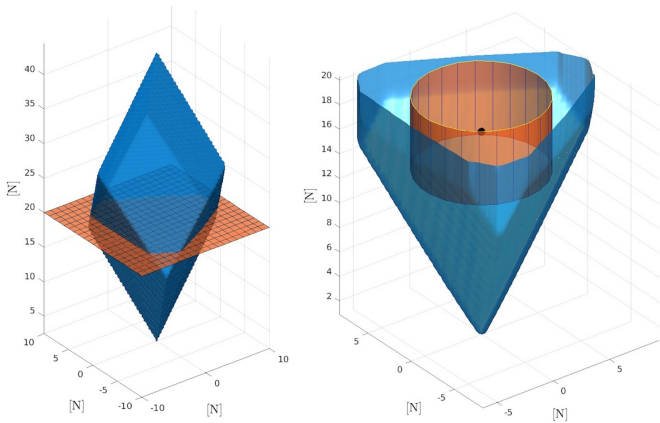

\centering
\figForceVolumeFull\hfill
\figForceVolume
\caption{Left: The blue volume encloses the set of feasible forces, obeying the constraints of minimal and maximum rotor spinning velocity for the Tilt-Hex. The red plane visualizes the cut of the sectional view of the plot on the right. Right: lower part of the cut of the left figure. The red cylinder visualizes the volume of the imposed cylindric force constraint. Notice that the cylinder is fully inside the volume of feasible forces. The black dot in the center visualizes the force needed to hover horizontally and therefore represents the nominal working point.}
\label{fig:cylinder_in_cone}
\end{figure}

The \SI{30.5}{\centi\meter} propeller blades are powered by MK3638 brushless motors provided by MikroKopter. The maximum lift force of a single motor-propeller combination is $\SI{12}{\newton}$. The electronic speed controller (ESC) driving the motor is a BL-Ctrl-2.0 from MikroKopter. The controller, running on the ESC, is an in-house development~\cite{2017c-FraMal} that controls the propeller rotational speed in closed loop at a variable frequency (e.g., when the propeller rotation speed is \SI{70}{\hertz} the control frequency is \SI{3.29}{\kilo\hertz}). An inertial measurement unit (IMU) provides accelerometer and gyroscope measurements ($\SI{500}{\hertz}$) and a marker-based motion capture system provides position and orientation measurements ($\SI{100}{\hertz}$) of the platform. The motion capture and IMU measurements are fused via a UKF state estimator to obtain the full state at $\SI{500}{\hertz}$. 

A PnP-algorithm and an onboard camera could easily be utilized to replace the motion capture system, as done for a similar system in~\cite{2017h-SanAreTogCamFra}, by some of the authors. However it is preferable to validate the proposed controller using the motion capture and the IMU because in this way we minimize the influence of the particular perception system used on the controller performances and the evaluation results more fair. 

Finally, the controller has been implemented in Matlab-Simulink and runs at $\SI{500}{\hertz}$ on a stationary base PC connected to the Tilt-Hex through a serial cable. As the computational effort of the controller is very low (considerably below \SI{1}{\milli\second} per control loop) it could be ported easily to an on-board system. Based on our experience with a similar porting, we expect the performances of an onboard implementation to be much better than the Matlab-Simulink implementation, thanks to the possibility of reaching a faster control frequency (larger than $\SI{1}{\kilo\hertz}$) and almost real-time capabilities (latency below \SI{1}{\milli\second}). Therefore the experiments proposed in Section~\ref{sec:ExperimentalResults} represent a worst case scenario from this point of view, in the sense that an onboard implementation can only perform better than what we tested.

\subsubsection{Experimental Validations}
\label{sec:ExperimentalResults}

Three different experimental validations have been conducted. The task of of all the experiments is to follow at best a given reference pose trajectory $\vect{q}_r(t)=(\vect{p}_r(t),\vect{R}_r(t))$, i.e., position and orientation in function of time~$t$. In the first experimental validation, called \emph{Experimental Batch~1} the value of $r_{xy}$ in~\eqref{eq:r_xy} and of $\vect{R}_r$ is kept constant. The Experimental Batch~1 is composed by three experiments: Exp.\,1.1, Exp.\,1.2 and Exp.\,1.3, which are detailed in the following.
Instead in the second validation, denoted with Exp.\,2,  $\vect{R}_r$ varies over time. Finally in the third validation, referred to as Exp.~3 is $r_{xy}$ that varies over time. 

For the reader's convenience the rotation matrices used internally by the controller have been converted to \textit{roll-pitch-yaw} angles, with the convention $\mathbf{R}_{\bullet} \to \phi_{\bullet}, \theta_{\bullet}, \psi_{\bullet}$ in all plots. However, the internal computations are all done with rotation matrices. In plots where data are very noisy a filtered version (darker color) is presented together with the original data (lighter color in background), as, e.g. in the last row of Fig.~\ref{fig:ExpOnePosition}.

The interested reader is referred to the multimedia attachment of this paper to fully enjoy the videos all the experiments. Furthermore, all the experimental data (with suitable scripts to plot them) are provided for download at  the the following link \url{http://homepages.laas.fr/afranchi/files/2017/dataset1.zip}.

\subsection{Experimental Batch~1}

In the Experimental Batch~1 the translational part of the reference trajectory, i.e., $\vect{p}_r(t)=[p_{r_x}(t)\,p_{r_y}(t)\,p_{r_z}(t)]^T$, is such that $p_{r_y}(t)=\SI{0}{\meter}$ and $p_{r_z}(t)=\SI{1}{\meter}$ %
remain constant over the whole trajectory, while  $p_{r_x}(t)$ oscillates sinusoidally between \SI{-1.2}{\meter} and \SI{1.2}{\meter}
with constant amplitude and time-varying frequency. The time-varying frequency is chosen such that the envelope of $\ddot{p}_{r_x}(t)$ is first  quasi-linearly increasing from \SI{0}{\meter\per\second\squared} up to \SI{5.9}{\meter\per\second\squared} and then quasi-linearly decreasing down to \SI{0}{\meter\per\second\squared} -- see Fig.~\ref{fig:ExpOnePosition}: first (top) and last (bottom) plot. The trajectory includes also a start-from-rest maneuver and a stop-to-rest maneuver at the beginning and the end, respectively.
The reference orientation is constant and horizontal during the whole trajectory, i.e., $\vect{R}_r(t)=\mat{I}_{3\times 3}$
The batch is composed by three experiments described in the following.

In Exp.\,1.1 a value of $r_{xy}=\SI{3}{\newton}$ has been selected, which fits well inside the actual maximum lateral force of the Tilt-Hex given its mass of $m=\SI{1.8}{\kilo\gram}$. 
This means that in the parts of the trajectory in which $|\ddot{p}_{r_x}|>\frac{\SI{3}{\newton}}{\SI{1.8}{\kilo\gram}}= \SI{1.66}{\meter\per\second\squared}$ we expect the controller to let the platform deviate from $\vect{R}_r$ in order to track the high lateral acceleration. On the other side we expect a good independent tracking of position and orientation when $\ddot{p}_{r_x}\leq \SI{1.66}{\meter\per\second\squared}$.
In fact, Exp.1.1 is meant to illustrate the canonical behavior of proposed controller when controlling a fully-actuated LBF platform. 

In Exp.\,1.2 we present the result of the controlled system in the extreme situation of $r_{xy}=\SI{0}{\newton}$ which resembles an underactuated aerial vehicle. Therefore we expect the controller to let the platform deviate almost always from $\vect{R}_r$. This experiment is meant to show that the proposed controller can handle the (classic) underactuated case,  thus not requiring the user to switch between  different controllers.

In Exp.\,1.3 we perform the same trajectory with the state-of-the-art controller presented in~\cite{2015e-RajRylBueFra}, a controller that does not take into account the input saturations. This experiment is meant to show how the proposed controller outperforms the state of the art controllers in terms of robustness and stability.

\subsubsection{Experiment 1.1}
\label{sec:exp1_1}

\begin{figure}[t]
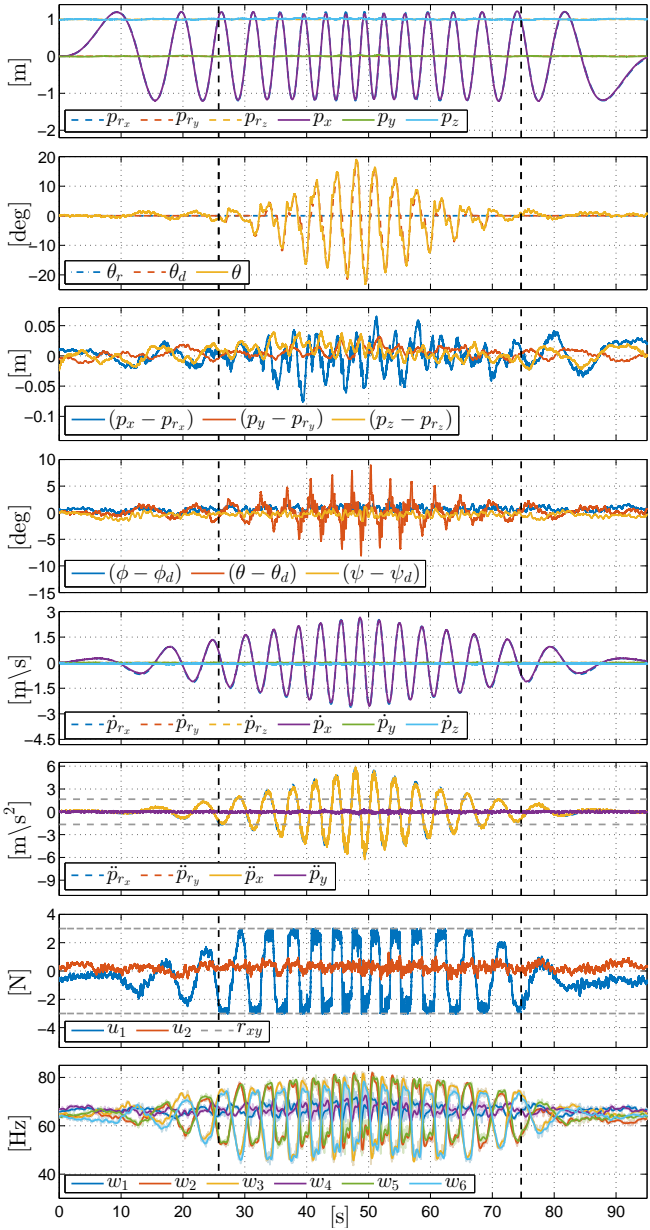

\centering
\figExpOnePosition\\
\figExpOneOrientation\\
\figExpOnePosError\\
\figExpOneRotError\\
\figExpOneVelocity\\
\figExpOneAcceleration\\
\figExpOneUOne\\
\figExpOneDesW
\caption{Exp.\,1.1: Desired position: sinusoidal motion along the $\vect{x}_W$ axis with constant amplitude and triangular (first increasing then decreasing) frequency. Desired orientation: constantly horizontal. Lateral force bound: constant $r_{xy}=\SI{3}{\newton}$.}
\label{fig:ExpOnePosition}
\end{figure}

Fig.~\ref{fig:ExpOnePosition} visualizes the main results of Exp\,1.1.
There are three clearly distinct temporal phases in the experiment separated by the vertical dashed lines in the second plot and defined by $t \in T_1 = [\SI{0}{\second},\SI{25.8}{\second}]$, $t \in T_2 = [\SI{25.8}{\second},\SI{74}{\second}]$, and $t \in T_3 = [\SI{74}{\second},\SI{95}{\second}]$, respectively.

In the first and third phases the full-pose reference trajectory is always feasible. In fact, the norm of the acceleration of the reference position trajectory $|\ddot{p}_{r_x}|$ is always below $\SI{1.66}{\meter\per\second\squared}$ (see the sixth plot- dashed grey lines). This value represents the maximum lateral acceleration that the controller can impose to the platform when the orientation is kept horizontal and the bound $r_{xy}=\SI{3}{\newton}$ and mass of $m=\SI{1.8}{\kilo\gram}$ are considered.  In the second phase, instead, the full-pose reference trajectory is not always feasible and in the middle of the second phase, in the neighborhood of $t=\SI{49}{\second}$, the trajectory is mostly unfeasible, since the lateral acceleration required has peaks of $\SI{5.9}{\meter\per\second\squared}$ ($3.5$ times the maximum lateral acceleration attainable while keeping a horizontal orientation).

Accordingly to what expected, in the `feasible' phases ($T_1$ and $T_3$) both the orientation and position tracking errors w.r.t. the reference full-pose trajectory, are relatively low. In particular we have that the pose error is in average zero, and $\|\vect{p}(t)-\vect{p}_r(t)\|<\SI{0.02}{\meter}$ and $|\theta(t)-\theta_r(t)|<\SI{1.7}{\degree}$. 
In the `unfeasible' phase ($T_2$), the position tracking error is still good ($\|\vect{p}(t)-\vect{p}_r(t)\|<\SI{0.06}{\meter}$) while the desired orientation $\mat{R}_d(t)$ sensibly deviates from the reference one $\mat{R}_r(t)$ with a peak overshoting \SI{20}{\degree} for $|\theta_d(t)-\theta_r(t)|$ (see second plot). In fact, tilting is the only way by which the platform can track the desired position, given the lateral force bounds.
It is interesting to note that not only the reference position is well tracked along the whole experiment, but also the translational velocity and acceleration are, as shown in the fifth plot and sixth plot, respectively.

\begin{figure}[t]
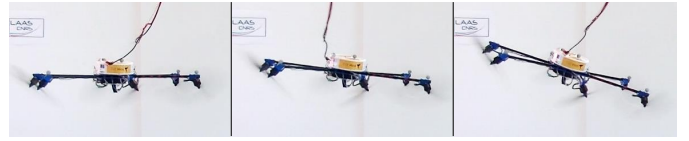

\centering
\figTiltHex
\caption{The Tilt-Hex performing Exp. 1.1 at different time instances: Left $t=\SI{15.4}{\second}$; middle $t=\SI{37.7}{\second}$; right $t=\SI{47.9}{\second}$. Although the reference orientation is constant and horizontal the Tilt-Hex adapts it orientation to allow for following the reference position.}
\label{fig:Tilt-Hex}
\end{figure}

Looking at the seventh plot one can appreciate how the controller keeps always the lateral force within the requested bounds and at the same time touches and stays on the bounds several times for several seconds. This is a clear index that the controller exploits at best the platform capabilities. Notice also how in the  phases in which the lateral force touches the bounds, the controller exploits the tilting of the platform (compare the second plot) in order to compensate for the partial loss of the lateral control authority and attain the force required to produce the needed acceleration.

Finally, for the sake of completeness, we present also the actual six rotor spinning frequencies $w_1,\dots, w_6$ in the eighth and final plot. 

Figure~\ref{fig:Tilt-Hex} shows the Tilt-Hex performing Exp.~1.1 in three different time instants. Furthermore, we encourage the reader to watch the multimedia atachment showing this and the other experiments.

\subsubsection{Experiment 1.2}

To test the behavior of the controller with an underactuated aerial system, in Exp.\,1.2 we ask the system to track the same trajectory of Exp\,1.1 but setting a zero maximum lateral force, i.e., $r_{xy}=\SI{0}{\newton}$. In this way the Tilt-Hex should emulate the behavior of an underactuated platform as, e.g., a collinear multi-rotor.
Figure~\ref{fig:ExpOnePosition_zero} shows the main plots, where we omit the plots that are similar to the ones in Fig.~\ref{fig:ExpOnePosition} in order to focus on the salient differences. Contrarily to Exp.\,1.1, in Exp.\,1.2 the first and third `feasible' phases do not exists. The whole experiment is a long unfeasible phase due to the constraint  $r_{xy}=\SI{0}{\newton}$, which makes impossible,  at any time, to track the constantly horizontal  reference orientation defined while following the sinusoidal reference position trajectory.

\begin{figure}[t]
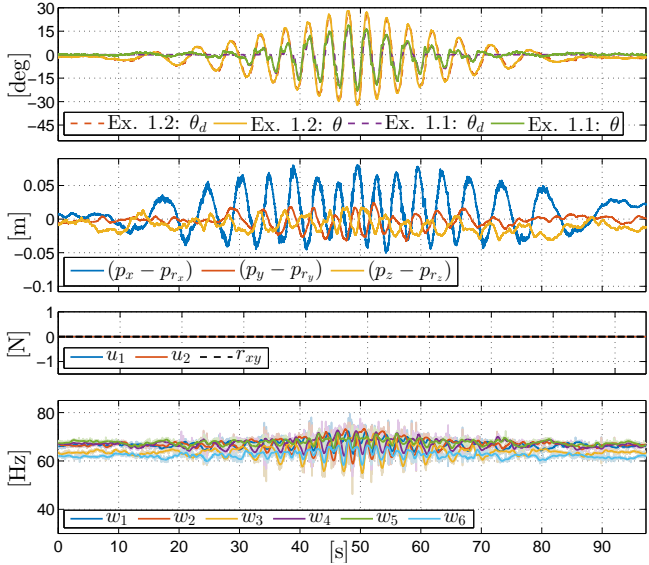

\centering
\figExpOneOrientationZero\\
\figExpOnePosErrorZero\\
\figExpOneUOneZero\\
\figExpOneDesWZero
\caption{Exp.\,1.2: Same desired trajectory as in Exp.1.1 but with $r_{xy}=0$.}
\label{fig:ExpOnePosition_zero}
\end{figure}

The orientation tracking of Exp.\,1.2 is presented in the first plot of Fig.~\ref{fig:ExpOnePosition_zero}, where also the one Exp.\,1.1 is shown again for ease of comparison. In the period of time $T_1$ defined in Sec.~\ref{sec:exp1_1}, the pitch tracking error in Exp.\,1.2, $|\theta(t)-\theta_r(t)|$, reaches $\SI{11}{\degree}$, i.e., $6.5$ times the peak of Exp.\,1.1 in the same period. 
In the period of time $T_2$ defined in Sec.~\ref{sec:exp1_1}, the pitch tracking error in Exp.\,1.2, $|\theta(t)-\theta_r(t)|$ reaches $\SI{31}{\degree}$ i.e., about $1.5$ times the peak of Exp.\,1.1 in the same period. 

Regarding the translation (second plot of  of Fig.~\ref{fig:ExpOnePosition_zero}), the peak of the position tracking error is about $3$ times larger (in the period $T_1$) and $1.4$ times larger (in the period $T_2$), when compared to the error peak of Exp.\,1.1 in the same periods. This is due to the fact that full actuation helps in minimizing the position tracking too (not only the orientation tracking).

In the third  plot of Fig.~\ref{fig:ExpOnePosition_zero} we can see that that the inputs $u_1$ and $u_2$ remain zero as expected during the full trajectory tracking, as required.
Finally, for completeness, we present also the actual six rotor spinning frequencies $w_1,\dots, w_6$ in the fourth and last plot of Fig.~\ref{fig:ExpOnePosition_zero}.
\subsubsection{Experiment 1.3}

In order to compare with the state-of-the-art methods such as~\cite{2015e-RajRylBueFra}, in Exp.\,1.3 we tested the controller with a saturated rotor spinning velocity with the minimum and maximum values in Exp.\,1.1 ($\SI{43}{\hertz}\le w_i \le\SI{83}{\hertz},$ $i\in [1\dots 6]$).  The results are depicted in Fig.~\ref{fig:ExpOnePositionInf}. The platform tracks well the reference trajectory till the output reaches its limit (t=\SI{34}{\second}) (see Fig.~\ref{fig:ExpOnePositionInf} - third plot, dashed line). Spinning velocities outside the limitation are asked by the controller and are therefore saturated.  This means that $\vect u_1$ and $\vect u_2$ in \eqref{eq:u1_ctrl}--\eqref{eq:u2_ctrl} cannot be generated anymore and the trajectory tracking performance decreases rapidly, until the system becomes completely unstable diverging from the reference position ($||\vect{p}-\vect{p}_r||>\SI{0.5}{\meter}$) and reference velocity ($||\dot{\vect{p}}-\dot{\vect{p}}_r||>\SI{1.2}{\meter\per\second}$) such that we had to abort the experiment. This experiment clearly shows how our controller outperforms a state-of-the-art controller in terms of performances and, most important of all, stability and safety.

\begin{figure}[t]
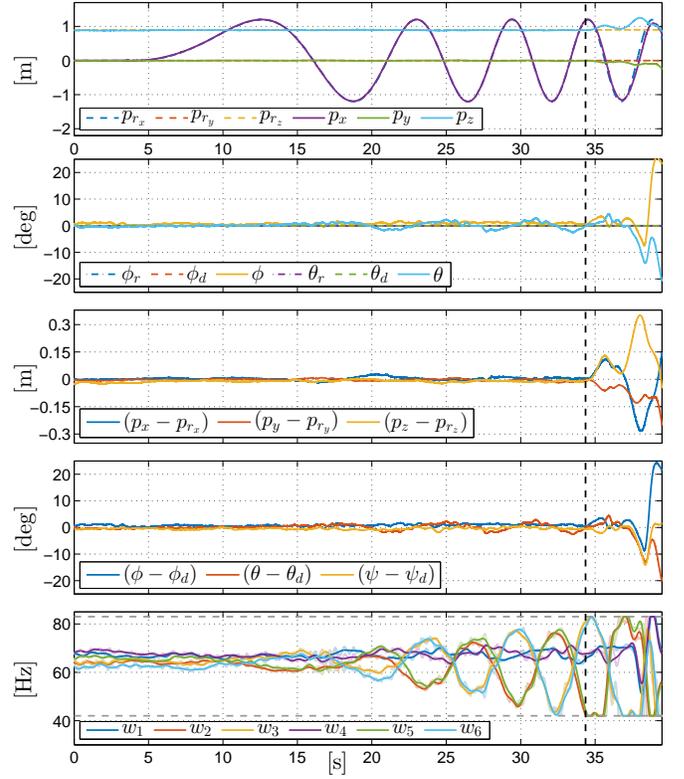

\centering
\figExpOnePositionSat\\
\figExpOneOrientationSat\\
\figExpOnePosErrorSat\\
\figExpOneRotErrorSat\\
\figExpOneDesWSat\\
\caption{Exp.\,1.3: Same desired trajectory as in Exp.1.1 but with saturated rotor spinning velocity $\vect w$ (saturation indicated by dashed grey lines in plot five). The experiment is automatically stopped after about \SI{39}{\second} because the system becomes unstable.}
\label{fig:ExpOnePositionInf}
\end{figure}

\subsection{Experiment 2}

To present the full capabilities of the full pose controller on LBF vehicles, in Exp.\,2 we set $\vect{p}_r(t)$ as in Exp.\,1.1, but we additionally ask the platform to follow a $\vect{R}_d(t)$ generated applying to $\vect{I}_{3\times 3}$ a sinusoidal rotation about the $\vect{y}_W$ axis (with an amplitude of $\SI{10}{\degree}$). This rotational motion is particularly chosen such that the orientation of the Tilt-Hex is in opposition of phase with respect to the orientation that an underactuated vehicle would need in order to track $\vect{p}_r(t)$ (i.e., the top part of the platform facing outwards at the two ends of the position trajectory, while for, e.g. a quadrotor the top would face always toward the center of the position trajectory). 
Also in this case the reference-to-actual position error and the desired-to-actual orientation error remain bounded and small (see the third and fourth plot of Fig.~\ref{fig:ExpOneTwoPosition}). The maximum lateral thrust is reached sooner than Exp.\,1.1 (at $t=\SI{10}{\second}$), due to the special inclination required to the Tilt-Hex. This results in an earlier adaptation of the desired orientation trajectory $\vect{R}_d$. As expected, at the time of highest accelerations ($\SI{45}{\second}\le t \le \SI{55}{\second}$) the desired pitch angle $\theta_d$ is almost inverted with respect to the reference angle $\theta_r$.

\begin{figure}[t]
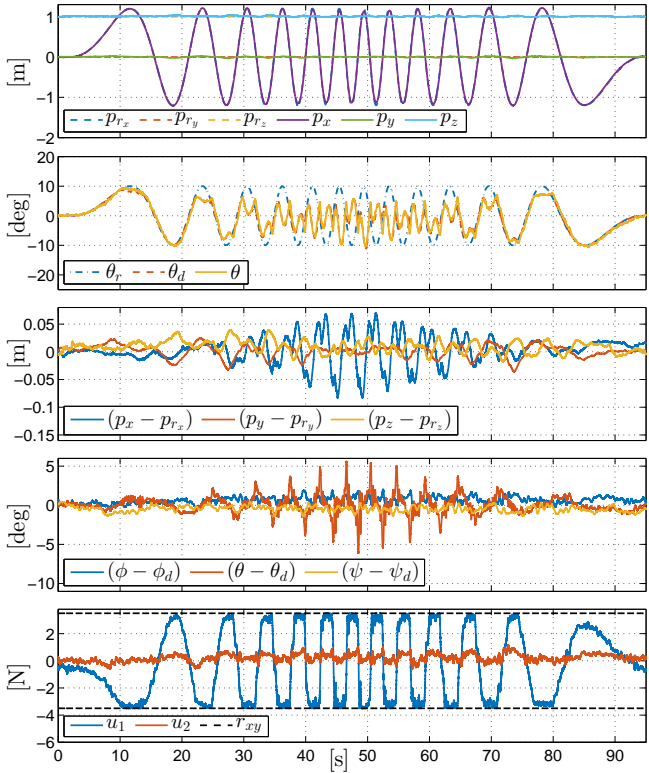

\centering
\figExpOneTwoPosition\\
\figExpOneTwoOrientation\\
\figExpOneTwoPosError\\
\figExpOneTwoRotError\\
\figExpOneUOneTwo\\
\caption{Exp.\,2: 
Desired position: as in Exp.\,1.1. Desired orientation: sinusoidal rotation about the $\vect{y}_W$ axis in opposition of phase w.r.t. a hypothetical quadrotor following the desired position. Lateral force bound: constant $r_{xy}=\SI{3}{\newton}$.}
\label{fig:ExpOneTwoPosition}
\end{figure}

\subsection{Experiment 3}

The conclusive Exp.\,3 has been designed to stress the fact that the presented controller can seamlessly work with \emph{under} and \emph{fully-actuated} platforms and moreover with platforms that can actively change between these two configurations during flight, as the one presented in~\cite{2016j-RylBicFra}. 
The plots of the experiment are reported in Fig.~\ref{fig:ExpTwoPosition}.

The reference trajectory in position for the second experiment consists out of two regular sinusoidal motions along the $\vect{x}_W$ and $\vect{y}_W$ axes with an amplitude of $\SI{1.3}{\meter}$ and $\SI{0.5}{\meter}$, respectively, and constant frequencies.
The first seconds of the translational trajectory are a polynomial trajectory to connect the initial state of the Tilt-Hex (hovering at an arbitrary position) with the sinusoidal reference trajectory. 
The reference orientation is constant $\mat{R}_r(t)=\mat{I}_{3\times 3}$.
The lateral force bound $r_{xy}$ is changed over time, in particular, it is $r_{xy}(t)=\SI{0}{\newton}$ for $t\in[\SI{0}{\second},\SI{18}{\second}]$, $r_{xy}(t)=\SI{10}{\newton}$ for $t\in[\SI{38}{\second},\SI{56}{\second}]$, and it is linearly increasing from \SI{0}{\newton} to \SI{10}{\newton} for 
$t\in[\SI{18}{\second},\SI{38}{\second}]$
The selection of the time-varying $r_{xy}(t)$ forces to system to be underactuated until $t=\SI{18}{\second}$.
 
As it should be, the position tracking is always good (first and third plots). However the system cannot track at the same time the reference position the reference orientation. 
Until $t=\SI{34}{\second}$ the system is partially fully actuated (w.r.t. the trajectory to be followed). The orientation tracking gradually improves. At $t=\SI{34}{\second},$ $r_{xy}$ is large enough to track the reference orientation at any time (see second plot).

The behavior of $r_{xy}$ is  visualized in the fifth plot, together with the actual lateral forces implemented by the controller, which are always kept within the bounds.
The fact that the lateral force bound changes over time does not deterioriate  the behavior of the controller, which is instead able to cope with the time-varying constraint exploiting the platform capability always at its best.

Finally it is worth to mention that the ranges of the propeller spinning frequencies  utilized by the controller ($w_1\dots w_6$) naturally increases with the increase of $r_{xy}(t)$ at the benefit of completely tracking the full pose reference trajectory (see the sixth plot).

\begin{figure}[t]
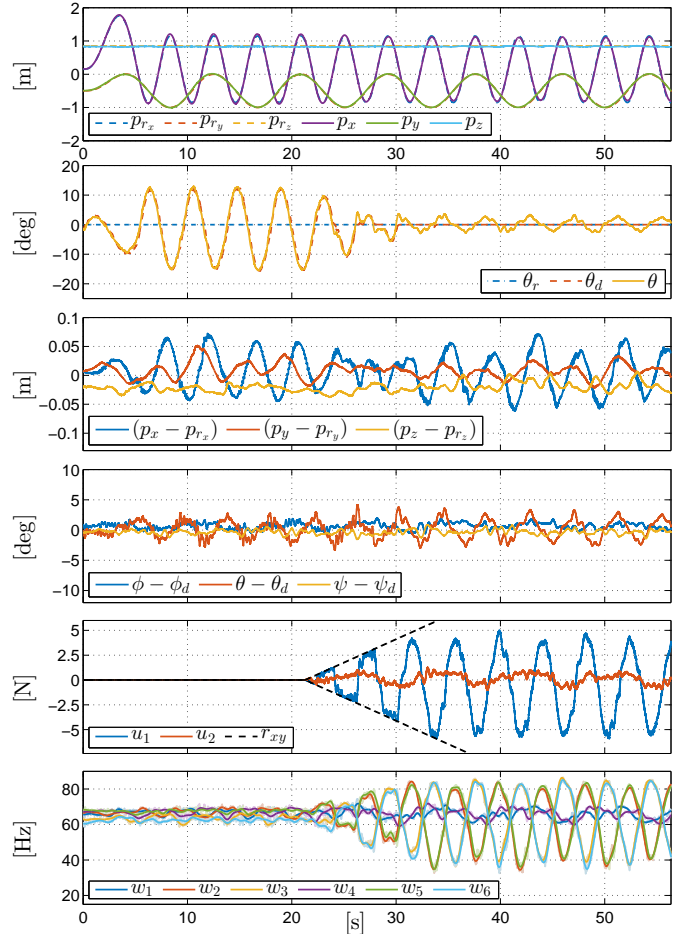

\centering
\figExpTwoPosition\\
\figExpTwoOrientation\\
\figExpTwoPosError\\
\figExpTwoRotError\\
\figExpTwoUOne\\
\figExpTwoDesW
\caption{Exp.\,3: Desired position: composition of two sinusoidal motions along $\vect{x}_B$ and $\vect{y}_B$, with constant amplitudes and frequencies.  Desired orientation: constant and horizontal. Lateral force bound: $r_{xy}$ linearly increasing from \SI{0}{\newton} to \SI{10}{\newton}.}
\label{fig:ExpTwoPosition}
\end{figure}

\section{Conclusions}\label{sec:concl}

In this paper we introduced the new class of Laterally-Bounded Force (LBF) aerial vehicles. This class is general enough to encompass a large variety of recently conceived aerial vehicles having the possibility to actuate the thrust in a direction other than the principal one. Common underactuated platforms are included in this class as a degenerate (but fully admissible) case. 

For this class of vehicles we proposed a geometric controller in SE(3) that is able to let it track any feasible full-pose (6D) trajectory. The controller adapts seamlessly to the case the trajectory is (or becomes) not feasible or that the platform is (or becomes) underactuated. Being defined in SE(3) the controller is not prone to the singularities of local chart orientation representations.

The practicability of the theory has been shown in real experiments. Furthermore this controller has been already used as inner loop controller for other projects involving also aerial physical interaction, as, e.g., in~\cite{2017e-RylMusPieCatAntCacFra}.

In future we plan to study adaptive and robust techniques to deal with parameter uncertainties and malfunction of some of the actuators.

\section{Appendix}\label{sec:appendix}

\begin{lem}\label{lem:cascade_system}
Consider the cascade system
\begin{align}\label{eq:cascade}
\dot{\mathbf{x}}_1(t) &= \mathbf{A} \mathbf{x}_1(t) + \mathbf{B} \mathbf{x}_2(t) \\
\dot{\mathbf{x}}_2(t) &= \mathbf{f}\left(\mathbf{x}_2(t)\right)\nonumber
\end{align}
where $\mathbf{A} \,\in\,\mathbb{R}^{n_1 \times n_1}$, $\mathbf{B} \in \mathbb{R}^{n_1 \times n_2}$ and $\mathbf{f} : \mathbb{R}^{n_2} \to \mathbb{R}^{n_2}$ is a continuous function. We assume that
\begin{enumerate}
\item $\mathbf{A}$ is a Hurwitz matrix;
\item there exists a compact set $\mathcal{X} \subseteq \mathbb{R}^{n_2}$ with $\mathbf{0} \in \mathcal{X}$, and suitable constants $C>0$, $\rho>0$, such that
\begin{equation}\label{eq:x_2}
\|\mathbf{x}_2(t)\| \leq C e^{-\rho \,t} \|\mathbf{x}_2(0)\|,
\end{equation}
for any $\mathbf{x}_2(0) \in \mathcal{X}$.
\end{enumerate}
Then, the zero equilibrium $(\mathbf{x}_1, \mathbf{x}_2)=(\mathbf{0},\mathbf{0})$ of the cascade system \eqref{eq:cascade} is exponentially stable. The basin of attraction is $\mathbb{R} \times \mathcal{X}$. 
\end{lem}

\begin{proof}
From the first equation in \eqref{eq:cascade}, we have that
$$
\mathbf{x}_1(t)= e ^{\mathbf{A}t}\mathbf{x}_1(0)+ \int_{0}^t e^{\mathbf{A}(t-\tau)}\mathbf{B}\mathbf{x}_2(\tau) \,d\tau.
$$
Since $\mathbf{A}$ is Hurwitz and since signal $\mathbf{x}_2$ satisfies \eqref{eq:x_2}, it follows that there exist suitable positive constants $C_1$, $C_2$, $\rho_1$, $\rho_2$ such that
\begin{align*}
\|\mathbf{x}_1(t)\| & \leq C_1 e^{-\rho_1 t} \|\mathbf{x}_1(0)\| + \left|\left| \int_{0}^t e^{\mathbf{A}(t-\tau)}\mathbf{B}\mathbf{x}_2(\tau) \,d\tau \right|\right| \\
&\leq C_1 e^{-\rho_1 t} \|\mathbf{x}_1(0)\| + C_2 e^{-\rho_2 t} \|\mathbf{x}_2(0)\|
\end{align*}
Defining $C_3=\max \left\{C_1, C_2\right\}$, $\rho_3=\min \left\{\rho_1, \rho_2\right\}$, yields
$$
\|\mathbf{x}_1(t)\|  \leq C_3 e^{-\rho_3 t} \left( \|\mathbf{x}_1(0)\| + \|\mathbf{x}_2(0)\| \right)
$$ 
and, in turn,
$$
\|\mathbf{x}_1(t)\|^2  \leq C_3^2 e^{-2\rho_3 t} \left( \|\mathbf{x}_1(0)\|^2 + 2\|\mathbf{x}_1(0)\|\|\mathbf{x}_2(0)\| + \|\mathbf{x}_2(0)\|^2 \right).
$$ 
Since $2\|\mathbf{x}_1(0)\|\|\mathbf{x}_2(0)\| \leq \|\mathbf{x}_1(0)\|^2 + \|\mathbf{x}_2(0)\|^2$ we have, 
$$
\|\mathbf{x}_1(t)\|^2  \leq 2C_3^2 e^{-2\rho_3 t} \left( \|\mathbf{x}_1(0)\|^2 +  \|\mathbf{x}_2(0)\|^2 \right).
$$
Now, combining the previous inequality with \eqref{eq:x_2}, we can write
\begin{align*}
&\|\mathbf{x}_1(t)\|^2 + \|\mathbf{x}_2(t)\|^2 \\
&\,\,\,\,\,\,\,\,\,\, \leq 2C_3^2 e^{-2\rho_3 t} \left( \|\mathbf{x}_1(0)\|^2 +  \|\mathbf{x}_2(0)\|^2 \right) + C^2 e^{-2\rho t} \|\mathbf{x}_2(0)\|^2\\
& \,\,\,\,\,\,\,\,\,\, \leq \left(2C_3^2 e^{-2\rho_3 t} + C^2 e^{-2\rho t} \right) \left(\|\mathbf{x}_1(0)\|^2 +  \|\mathbf{x}_2(0)\|^2\right)
\end{align*}
Let $\mathbf{x}(t)= \left[\mathbf{x}_1(t) \,\,\,\,\mathbf{x}_2(t)\right]^T$ and let $\rho_4= \min \left\{\rho_3,\rho\right\}$ and $C_4=\sqrt{2C_3^2+C^2}$. Then
$$
\|\mathbf{x}(t)\|^2 \leq C_4^2 e^{-2\rho_4t} \|\mathbf{x}(0)\|^2,
$$
and, equivalently,
$$
\|\mathbf{x}(t)\| \leq C_4 e^{-\rho_4t} \|\mathbf{x}(0)\|.
$$
This shows that systems in \eqref{eq:cascade} is exponentially stable.
\end{proof}

\bibliographystyle{IEEEtran}
\bibliography{./alias,./main,./bibCustom}

\end{document}